\documentclass[reqno]{amsart}
\usepackage{}
\usepackage{mathrsfs}
\usepackage{color}
\usepackage{amsmath}
\usepackage{amsfonts}
\usepackage{amssymb}
\usepackage{graphicx}%
\usepackage{hyperref}

 \newtheorem{Theorem}{Theorem}[section]
 \newtheorem{Corollary}{Corollary}[section]
 \newtheorem{Lemma}{Lemma}[section]
 \newtheorem{Proposition}{Proposition}[section]

 \newtheorem{Problem}{Problem}[section]

 \newtheorem{Remark}{Remark}[section]

 \numberwithin{equation}{section}

\begin{document}

\title[Multiplier ideal sheaves, jumping numbers, restriction formula]
 {Multiplier ideal sheaves, jumping numbers, and the restriction formula}

\author{Qi'an Guan}
\address{Qi'an Guan: School of Mathematical Sciences, and Beijing International Center for Mathematical Research,
Peking University, Beijing, 100871, China.}
\email{guanqian@amss.ac.cn}
\author{Xiangyu Zhou}
\address{Xiangyu Zhou: Institute of Mathematics, AMSS, and Hua Loo-Keng Key Laboratory of Mathematics, Chinese Academy of Sciences, Beijing, China}
\email{xyzhou@math.ac.cn}

\thanks{The authors were partially supported by NSFC-11431013. The second author would like to thank NTNU for offering him Onsager Professorship.
The first author was partially supported by NSFC-11522101.}



\date{\today}

\dedicatory{}

\commby{}


\begin{abstract}
In the present article,
we establish an equality condition in the restriction formula on jumping numbers
by giving a sharp lower bound of the dimension of the support of a related
coherent sheaf.
As applications,
we obtain equality conditions in
the restriction formula on complex singularity exponents
by giving the dimension, the regularity and the transversality of the support,
and we also obtain some sharp equality conditions in the fundamental subadditivity property on complex singularity exponents.
We also obtain two sharp relations on jumping numbers.
\end{abstract}

\maketitle
\section{Backgrounds and Motivations}

Multiplier ideal sheaves associated to plurisubhamonic functions and their associated invariants (say, complex singularity exponents i.e. log canonical threshold (lct) in algebraic geometry and jumping numbers) have become in recent years a fundamental tool in several complex variables and algebraic geometry, and have been developing with great success by many mathematicians
(see e.g. \cite{Nadel90,siu96,demailly-note2000,Blic-Lazar03,Lazar04I,siu05,siu09}).

Various important and fundamental properties about the multiplier ideal sheaves and the invariants
have been established, such as the first properties: e.g., coherence, integrally closedness, Nadel vanishing theorem; and further properties: e.g., the restriction formula and subadditivity property
(see e.g. \cite{DEL00,D-K01,demailly2010}). Very recently, strong openness property of the multiplier ideal sheaf is established by our solution of Demailly's strong openness conjecture (see e.g. \cite{GZopen-c}).

In the present article, we'll discuss the restriction formulas for multiplier ideal sheaves and on jumping numbers, and related subadditivity property. Based on the strong openness property and some other recent results, we establish sharp equality conditions in the restriction formula and subadditivity property in Demailly-Ein-Lazaseld's paper \cite{DEL00} and Demailly-Koll\'{a}r's paper \cite{D-K01},
by giving sharp lower bounds of the dimensions of the support of the related coherent analytic sheaves.
We also discuss some new properties about the multiplier ideal sheaves.

\subsection{Organization of the paper}
$\\$

In the present section,
we recall the backgrounds and the motivations of the problems
about sharp equality conditions in the restriction formula on jumping numbers and
the fundamental subadditivity property on complex singularity exponents
(Problem \ref{prob:jump_equality},
Problem \ref{prob:A_dim} and Problem \ref{prob:equality_subadd_jump}).

In Section \ref{sec:main_result}, we
present the main results of the present paper:
the solution of Problem \ref{prob:jump_equality}
(Theorem \ref{t:jump_equality}, main theorem),
the solutions of
Problem \ref{prob:A_dim} and Problem \ref{prob:equality_subadd_jump}
(Theorem \ref{thm:jump_equality_dim_sing}, Theorem \ref{thm:lct_slice_graph} and Theorem \ref{thm:equality_subadd_jump}, applications of Theorem \ref{t:jump_equality});
two sharp relations on jumping numbers (Corollary \ref{coro:GZ_jump_sharp} of Theorem \ref{t:GZ_jump_sharp} and Theorem \ref{thm:sharp_jump_inequ}) and the slicing result on complex singularity exponents (Remark \ref{rem:lct_slice}).
In Section \ref{sec:preparetory},
we recall or give some preliminary results used in the proof of the main theorem and applications.
In Section \ref{sec:proof_main},
we prove the main theorem (Theorem \ref{t:jump_equality}).
In Section \ref{sec:proof_applications},
we prove the applications of the main theorem
(Theorem \ref{thm:jump_equality_dim_sing}, Theorem \ref{thm:lct_slice_graph},
Remark \ref{prop:add_dim_nonregular} and Proposition \ref{prop:lct_add_graph}).
In Section \ref{sec:proof_relations},
we prove the two sharp relations on jumping numbers (Corollary \ref{coro:GZ_jump_sharp} of Theorem \ref{t:GZ_jump_sharp} and Theorem \ref{thm:sharp_jump_inequ}).
In section \ref{sec:Berndtsson},
we present a relationship between the fibrewise Bergman kernels and integrability.

\subsection{Restriction formula and subadditivity property}
$\\$

Let $\Omega$ be a domain in $\mathbb{C}^{n}$ with coordinates $(z_{1},\cdots,z_{n})$ and origin $o=(0,\cdots,0)\in\Omega$.
Let $u$ be a plurisubharmonic function on $\Omega$.
Nadel \cite{Nadel90} introduced
the multiplier ideal sheaf $\mathcal{I}(u)$
which can be defined as the sheaf of germs of holomorphic functions $f$ such that
$|f|^{2}e^{-2u}$ is locally integrable.
Here $u$ is regarded as the weight of $\mathcal{I}(u)$.

It is well-known that the multiplier ideal sheaf $\mathcal{I}(u)$ is coherent and integral closed, satisfies Nadel's vanishing theorem \cite{Nadel90} and the restriction formula and subadditivity property,
and the strong openness property $\mathcal{I}(u)=\cup_{\varepsilon>0}\mathcal{I}((1+\varepsilon)u)$ \cite{GZopen-a,GZopen-b,GZopen-c}
i.e. our solution of Demailly's strong openness conjecture
(the background and motivation of the conjecture could be referred to \cite{demailly-note2000,demailly2010}).

Let $I\subseteq \mathcal{O}_{o}$ be a coherent ideal.
The jumping number $c_{o}^{I}(u)$ is defined
(see e.g. \cite{JM12,JM13})
$$c_{o}^{I}(u):=sup\{c\geq0:|I|^{2}\exp{(-2cu)}\,\,\text{is integrable near}\,\,o\},$$
which can be reformulated by $c_{o}^{I}(u):=sup\{c\geq0:\mathcal{I}(cu)_{o}\supseteq I\}.$

Especially, when $I=\mathcal{O}_{o}$, the jumping number is just the complex singularity exponent denoted by $c_{o}(u)$
(see \cite{tian87}, see also \cite{demailly-note2000,demailly2010}) (or log canonical threshold by algebraic geometer see \cite{Sho92,Ko92}).

By Berndtsson's solution (\cite{berndtsson13}) of the openness conjecture $\mathcal{I}(c_{o}(u)u)_{o}\neq \mathcal{O}_{o}$ posed in \cite{D-K01},
it follows that $\{z|c_{z}(u)\leq c_{o}(u)\}=Supp(\mathcal{O}/\mathcal{I}(c_{o}(u)u))$, which is an analytic set since $\mathcal{I}(c_{o}(u)u)$ is a coherent ideal sheaf \cite{Nadel90} and the support of a coherent analytic sheaf is analytic.

\

Let $\mathcal{F}\subseteq\mathcal{O}$ be a coherent ideal sheaf.
By the definition of $c_{z}^{\mathcal{F}_{z}}(u)$ and the strong openness property,
it follows that
$c_{z}^{\mathcal{F}_{z}}(u)>p \Longrightarrow \mathcal{F}_{z}\subseteq\mathcal{I}(pu)_{z}$ and
$c_{z}^{\mathcal{F}_{z}}(u)\leq p \Longrightarrow \mathcal{F}_{z}\not\subseteq\mathcal{I}(pu)_{z}$.

Combining the fact that the support of a coherent analytic sheaf is an analytic subset,
one obtains

\emph{The lowerlevel set of jumping numbers $\{z|c_{z}^{\mathcal{F}_{z}}(u)\leq p\}=Supp(\mathcal{F}/(\mathcal{F}\cap\mathcal{I}(pu)))$
is an analytic subset.}

Let $H=\{z_{k+1}=\cdots=z_{n}=0\}$.
In \cite{DEL00} (see also (14.1) in \cite{demailly2010}),
the following restriction formula for multiplier ideal sheaves
has been stated by rephrasing Ohsawa-Takegoshi $L^{2}$ extension theorem:\\

\textbf{Restriction formula (for multiplier ideal).} $\mathcal{I}(u|_{H})\subseteq \mathcal{I}(u)|_{H}.$\\

Using the strong openness property,
it follows that the above
restriction formula for multiplier ideal is equivalent to\\

\textbf{Restriction formula (on jumping number).}
Let $I$ be a coherent ideal on $\mathcal{O}_{o'}$, where $o'$ is the origin in $H$. Then
$c_{o'}^{I}(u|_{H})\leq\sup\{c_{o}^{\tilde{I}}(u)|\tilde{I}\subseteq\mathcal{O}_{o}\ \& \ \tilde{I}|_{H}=I\},$
where $o'$ emphasizes that $c_{o'}^{I}(u|_{H})$ is computed on the submanifold $H$.
\\

When $I=\mathcal{O}_{o'}$,
the restriction formula about jumping numbers degenerates to the following restriction formula (an "important monotonicity result" as said in \cite{D-K01}) about complex singularity exponents:

\begin{Proposition}
\label{prop:DK2000}\cite{D-K01}$c_{o'}(u|_{H})\leq c_{o}(u)$, where $u|_{H}\not\equiv-\infty$.
\end{Proposition}

In \cite{D-K01} (see also (13.17) in \cite{demailly2010}),
the following fundamental subadditivity property of complex singularity exponents has been presented:
\begin{Theorem}
\label{thm:subadd_cse}\cite{D-K01}
Let $I$ and $J$ be coherent ideals on $\mathcal{O}_{o}$.
Let $u=\log|I|$ and $v=\log|J|$.
$c_{o}(\max\{u,v\})\leq c_{o}(u)+c_{o}(v).$
\end{Theorem}

Let $o_{1}\in\Omega_{1}$ and $o_{2}\in\Omega_{2}$,
and let $\pi_{i}:\Omega_{1}\times\Omega_{2}\to\Omega_{i}$ be projections for $i\in\{1,2\}$.
Motivated by the proof of Theorem \ref{thm:subadd_cse} in \cite{D-K01} (see also (13.17) in \cite{demailly2010})
and using Theorem \ref{thm:DK_GZ} in \cite{GZopen-b,GZopen-effect} (i.e., our solution of a conjecture posed by Demailly and Kollar in \cite{D-K01}),
we obtain

\begin{Proposition}
\label{thm:add_prod_cse}
Let $I_{1}$ and $I_{2}$ be coherent ideals in $\mathcal{O}_{o_{1}}$ and $\mathcal{O}_{o_{2}}$ respectively,
$u$ and $v$ be plurisubharmonic functions near $o_{1}$ and $o_{2}$ respectively, then one has
$$c_{(o_{1},o_{2})}^{I_{1}\times I_{2}}(\max\{u\circ\pi_{1}, v\circ\pi_{2}\})= c_{o_{1}}^{I_{1}}(u)+c_{o_{2}}^{I_{2}}(v).$$
\end{Proposition}
Details of the proof of Proposition \ref{thm:add_prod_cse} is in subsection \ref{subsec:add_prod}.

Let $I_{1}=\mathcal{O}_{o}$ and $I_{2}=\mathcal{O}_{o}$.
Using Proposition \ref{prop:DK2000},
we generalize Theorem \ref{thm:subadd_cse} as follows

\begin{Theorem}
\label{thm:subadd_cse_general}
Let $u$ and $v$ be plurisubharmonic functions on $\Delta^{n}$.
Then
$$c_{o}(\max\{u,v\})\leq c_{o}(u)+c_{o}(v).$$
\end{Theorem}

\subsection{Problems about sharp equality conditions}
$\\$

Let $n\geq 2$.
Let $u=\log(\sum_{1\leq j\leq l}|z_{j}|^{2})^{1/2}.$
Note that $l>k\Rightarrow c_{o}(u)=l>k=c_{o}(u|_{H})$,
and $l\leq k\Rightarrow c_{o}(u)=l=c_{o'}(u|_{H})$.
Then it is natural to consider the following problem
about the sharp equality condition in the restriction formula on jumping numbers:

\begin{Problem}
\label{prob:jump_equality}
Let $I$ be a coherent ideal on $\mathcal{O}_{o'}$.
Suppose that
\begin{equation}
\label{equ:jump_equ}
c_{o'}^{I}(u|_{H})=\sup\{c_{o}^{\tilde{I}}(u)|\tilde{I}\subseteq\mathcal{O}_{o}\ \& \ \tilde{I}|_{H}=I\}=:c.
\end{equation}
Let $A=Supp(\mathcal{O}/\mathcal{I}(cu))$.
Can one obtain that
\begin{equation}
\label{equ:jump_equ_condition}
dim_{o}A\geq n-k?
\end{equation}
\end{Problem}

For the case $I=\mathcal{O}_{o'}$ and $(k,n)=(1,2)$, Problem \ref{prob:jump_equality} was solved by
Blel-Mimouni (\cite{B-M}) and Favre-Jonsson (\cite{FM05j}).

For the case $I=\mathcal{O}_{o'}$ and $(k,n)=(1,n)$,
Problem \ref{prob:jump_equality} was solved in \cite{GZopen-lelong}.

Recently, combining with the recent result in \cite{demailly-Pham}
and Proposition \ref{Pro:GZ1005} in \cite{GZopen-lelong},
Rashkovskii \cite{Rash1501} reproved the above result in \cite{GZopen-lelong}.

It is natural to ask whether the more precise condition $\dim_{o}A=n-k+\dim_{o}(A\cap H)$ holds?
However, the following example tells us that the above condition may not hold for general $I$.

\textbf{Example 1.} When $n=4$, $k=3$, $H=\{z_{4}=0\}$, $I=(z_{1})_{o'}$, $u=2\log(|z_{2}|+|z_{3}|)+2\log(|z_{1}|+|z_{4}|)$,
then $c^{I}_{o'}(u|_{H})=1=\sup_{\tilde{I}}\{c_{o}^{\tilde{I}}(u)\}$, $A=(\{z_{2}=z_{3}=0\}\cup\{z_{1}=z_{4}=0\})$,
and $\dim_{o}A=2<3=4-3+2=n-k+\dim_{o}(A\cap H)$.

It is known that for the case $k=1$ and any $n$,
one can obtain the regularity of $(A,o)$ (see \cite{GZopen-lelong}).
Then it is natural to consider the regularity of $(A,o)$ for general $k$.
However, the following example tells us that the regularity may not hold for general $I$:

\textbf{Example 2.} When $n=2$, $k=1$, $H=\{z_{2}=0\}$, $I=(z_{1})_{o'}$, $u=\log|z_{1}|+\log|z_{1}+z_{2}|$,
then $c^{I}_{o'}(u|_{H})=1=\sup_{\tilde{I}}\{c_{o}^{\tilde{I}}(u)\}$, $A=(\{z_{1}=0\}\cup\{z_{1}+z_{2}=0\})$,
$A\cap H=\{o\}$, and $(A,o)$ is not regular.
When $n=2$, $u=\log|z_{2}-z_{1}^{2}|$ and $(H=\{z_{2}=0\})$, one can obtain that $c_{o'}(u|_{H})=1/2<1=c_{o}(u)$.

By Example 1 and Example 2, it is natural to ask

\begin{Problem}
\label{prob:A_dim}
Assume that $c_{o'}(u|_{H})=c_{o}(u)$.

(1) Can one obtain $\dim_{o}A=n-k+\dim_{o}(A\cap H)$?

(2) If $(A\cap H,o)$ is regular,
can one obtain that $(A,o)$ is regular and $\dim(T_{A,o}+T_{H,o})=n$?
\end{Problem}

Note that
$\\$
(a) if $u=v=\log|z|$, then $c_{o}(\max\{u,v\})=\frac{1}{n}=\frac{2}{n}<c_{o}(u)+c_{o}(v)$;
$\\$
(b) if $n\geq2$, $u=\log|z'|$ and $v=\log|z''|$, then $c_{o}(\max\{u,v\})=n=k+(n-k)=c_{o}(u)+c_{o}(v)$,
where $z'=(z_{1},\cdots,z_{k})$ and $z''=(z_{k+1},\cdots,z_{n})$.

Therefore it is natural to consider the following problem about the sharp equality condition in
the generalized version of the fundamental subadditivity property of complex singularity exponents:
\begin{Problem}
\label{prob:equality_subadd_jump}
Let $u$ and $v$ be plurisubharmonic functions on $\Delta^{n}$.
Let $c=c_{o}(u)+c_{o}(v)$.
Let $A_{1}=V(\mathcal{I}(cu))$ and $A_{2}=V(\mathcal{I}(cv))$.
If
\begin{equation}
\label{equ:subadd_equ}
c_{o}(\max\{u,v\})=c,
\end{equation}
can one obtain that
\begin{equation}
\label{equ:subadd_approx}
dim_{o}A_{1}+dim_{o}A_{2}\geq n?
\end{equation}
\end{Problem}

\section{Main results and applications}\label{sec:main_result}

In this section, we
present the main results of the present paper.

\subsection{Main theorem: the solution of Problem \ref{prob:jump_equality}}
$\\$

In this section,
we solve Problem \ref{prob:jump_equality}

\begin{Theorem}
\label{t:jump_equality}(main theorem)
Suppose that equality \ref{equ:jump_equ} holds.
Then inequality \ref{equ:jump_equ_condition} holds.
\end{Theorem}

\begin{Remark}
Note that the points in $A\cap H$ are not considered
in the proof of Theorem \ref{t:jump_equality},
then we obtain a more subtle conclusion:
$$dim_{o}(A\setminus H)\geq n-k.$$
\end{Remark}

Let $I=\mathcal{O}_{o'}$, then we obtain the following corollary of Theorem \ref{t:jump_equality}

\begin{Corollary}
\label{coro:lct_restrict_equ}
Suppose that $c_{o'}(u|_{H})=c_{o}(u)=:c.$
Then we have
$$dim_{o}A\geq n-k.$$
\end{Corollary}

When $k=1$, one can obtain that
$$A=\{z|c_{z}(u)\leq c_{o}(u)\}$$
is regular at $o$ by Siu's decomposition of positive closed currents
(\cite{siu74}, see also \cite{demailly-book,demailly2010}).
For details we refer to \cite{GZopen-lelong}.

However, when $k>1$ and $n>2$,
$\{z|c_{z}(u)\leq c_{o}(u)\}$ may not be regular at $o$,
e.g., let $u:=\log|z_{1}|+\log|z_{2}|$, $(k,n)=(2,3)$ and $H=\{z_{3}=0\}$,
then $\{z|c_{z}(u)\leq c_{o}(u)\}=(\{z_{1}=0\}\cup\{z_{2}=0\})$,
and $c_{o}(u)=c_{o'}(u|_{H})=1$.

\subsection{Applications of the main theorem: the solutions of Problem \ref{prob:A_dim} and Problem \ref{prob:equality_subadd_jump}}
$\\$

Using Corollary \ref{coro:lct_restrict_equ},
we give an affirmative answer to  Problem \ref{prob:A_dim} (1) by the following general result
\begin{Theorem}
\label{thm:jump_equality_dim_sing}
Assume that $\dim((A\cap H,o)\setminus(V(I),o))=\dim_{o}(A\cap H)$.
If equality \ref{equ:jump_equ} holds,
then $\dim_{o}A=n-k+\dim_{o}(A\cap H)$.
Especially
if $c_{o'}(u|_{H})=c_{o}(u)$,
then $\dim_{o}A=n-k+\dim_{o}(A\cap H)$.
\end{Theorem}

Using Theorem \ref{thm:jump_equality_dim_sing},
we give an affirmative answer to Problem \ref{prob:A_dim} (2).
\begin{Theorem}
\label{thm:lct_slice_graph}
Let $H=\{z_{k+1}=\cdots=z_{n}=0\}$.
If $c_{o'}(u|_{H})=c_{o}(u)$,
then the following statements are equivalent

$(1)$ $(A\cap H,o)$ is regular;

$(2)$ there exist coordinates $(w_{1},\cdots,w_{k},z_{k+1},\cdots,z_{n})$ near $o$ and $l\in\{1,\cdots,k\}$
such that $(A,o)=(w_{1}=\cdots=w_{l}=0,o)$;

$(3)$ there exist coordinates $(w_{1},\cdots,w_{k},z_{k+1},\cdots,z_{n})$ near $o$ and $l\in\{1,\cdots,k\}$
such that $\mathcal{I}(c_{o}(u)u)_{o}=(w_{1},\cdots,w_{l})_{o}$.
\end{Theorem}

In the following part of this subsection,
we solve Problem \ref{prob:equality_subadd_jump}.

Denote by $H$ the diagonal of $\Delta^{n}\times\Delta^{n}$.
Using Proposition \ref{thm:add_prod_cse} ($\Omega_{1}\sim\Delta^{n}$, $\Omega_{2}\sim\Delta^{n}$, $I_{1}\sim\mathcal{O}_{o}$, $I_{2}\sim\mathcal{O}_{o}$)
and Corollary \ref{coro:lct_restrict_equ} ($u\sim \max\{u\circ\pi_{1},v\circ\pi_{2}\}$, $n\sim 2n$, $k\sim n$),
where $\sim$ means that the former is replaced by the latter, we obtain that
$$dim_{(o,o)}Supp(\mathcal{O}/\mathcal{I}(c\max\{u\circ\pi_{1},v\circ\pi_{2}\})\geq n.$$
Note that
$$Supp(\mathcal{O}/\mathcal{I}(c\max\{u\circ\pi_{1},v\circ\pi_{2}\})\subseteq Supp(\mathcal{O}/\mathcal{I}(cu\circ\pi_{1}))\cap Supp(\mathcal{O}/\mathcal{I}(cv\circ\pi_{2})),$$
then we give an affirmative answer to Problem \ref{prob:equality_subadd_jump}:
\begin{Theorem}
\label{thm:equality_subadd_jump}
If equality \ref{equ:subadd_equ} holds,
then inequality \ref{equ:subadd_approx} holds.
\end{Theorem}

Using Theorem \ref{thm:jump_equality_dim_sing},
we present the following remark of Theorem \ref{thm:equality_subadd_jump}

\begin{Remark}
\label{prop:add_dim_nonregular}
If equality \ref{equ:subadd_equ} holds,
then $\dim_{o}A_{1}+\dim_{o}A_{2}\geq n+\dim_{o}B$,
where $B=\{z|c_{z}(u)+c_{z}(v)\leq c\}$
is an analytic subset on $A_{1}\cap A_{2}$
(see subsection \ref{sec:add_nonregular}).
\end{Remark}

Let $n=2$, $u=\log|z_{1}|$, $v=\log|z_{1}-z_{2}^{2}|$.
As
$(|z_{1}|+|z_{2}^{2}|)/6\leq\max\{|z_{1}|,|z_{1}-z_{2}^{2}|\}\leq6(|z_{1}|+|z_{2}^{2}|)$,
it is clear that $c_{o}(\max\{u,v\})=1+1/2<2=c$,
$A_{1}\cap A_{2}=\{o\}$.
Then it is natural to consider the transversality between $A_{1}$ and $A_{2}$.

Using Theorem \ref{thm:lct_slice_graph},
we present
the following sharp equality condition in Theorem \ref{thm:subadd_cse_general}
by giving the regularity of $A_{1}$ and $A_{2}$ and the transversality between $A_{1}$ and $A_{2}$.

\begin{Proposition}
\label{prop:lct_add_graph}
Assume that $(A_{1},o)$ and $(A_{2},o)$ are both irreducible
such that $(B,o)=(A_{1}\cap A_{2},o)$, which is regular.
If equality \ref{equ:subadd_equ} holds,
then both $(A_{1},o)$ and $(A_{2},o)$ are regular such that $\dim (T_{A_{1},o}+T_{A_{2},o})=n$.
\end{Proposition}

\subsection{Two sharp relations on jumping numbers and application}

\subsubsection{A sharp upper bound of jumping numbers}
$\\$

Let $I\subseteq\mathcal{O}_{o}$ and $IJ\subseteq\mathcal{I}(c^{I}_{o}(u)u)_{o}$
$(\Leftrightarrow c^{IJ}_{o}(u)>c^{I}_{o}(u))$ be coherent ideals.
Using the strong openness property,
we obtain the following inequality on jumping numbers:

\begin{Theorem}
\label{t:GZ_jump_sharp}
$\frac{c^{I}_{o}(u)}{c^{IJ}_{o}(u)-c^{I}_{o}(u)}\geq c_{o}^{I}(\log|J|).$
\end{Theorem}

Given a coherent ideal $J\subseteq \mathcal{O}_{o}$. Letting $I=\mathcal{O}_{o}$,
we obtain the following sharp upper bound of
the jumping numbers represented by the complex singularity exponents:

\begin{Corollary}
\label{coro:GZ_jump_sharp}
$c^{J}_{o}(u)\leq\frac{c_{o}(u)}{c_{o}(\log|J|)}+c_{o}(u).$
\end{Corollary}

The following remark illustrates the sharpness of Corollary \ref{coro:GZ_jump_sharp}:

\begin{Remark}
\label{rem:sharp_inequ}
Let $(z_{1},\cdots,z_{n})$ be the coordinates of $\mathbb{C}^{n}$.
Suppose $u:=c\log|z_{n}|$ and $J=(z_{n}^{k})$.
Then we have $c^{J}_{o}(u)=\frac{k+1}{c}$, $c_{o}(\log|J|)=\frac{1}{k}$
and $c_{o}(u)=\frac{1}{c}$.
This gives the sharpness of Corollary \ref{coro:GZ_jump_sharp}.

More general,
when $J$ is principal ideal
(i.e., $J=(f)_{o}$),
and $u=c\log|f|$,
then we have $c^{J}_{o}(u)=\frac{1+c_{o}(\log|f|)}{c}$ and $c_{o}(u)=\frac{c_{o}(\log|f|)}{c}$.
This implies the sharpness of Corollary \ref{coro:GZ_jump_sharp}.
\end{Remark}

\subsubsection{A sharp inequality for jumping numbers and their slice restrictions}
$\\$

In this subsection, we present the following sharp inequality for jumping numbers and their restrictions on hyperplanes.

\begin{Theorem}
\label{thm:sharp_jump_inequ}
Let $\varphi$ be a plurisubharmonic function near the $o\in\mathbb{C}^{n}$,
and $h:=z_{n}$.
Let $H:=\{z_{n}=0\}$, and let $I\subset\mathcal{O}_{o'}$ be a coherent ideal,
where $o'$ is the origin in $H$.
Let $b_{0}:=\sup\{c_{o}^{\tilde{I}}(\varphi)|\tilde{I}\subseteq\mathcal{O}_{o}\& \tilde{I}|_{H}=I\}$,
and let $b_{1}:=\inf\{c_{o}^{\tilde{I}h}(\varphi)-c_{o}^{\tilde{I}}(\varphi)|\tilde{I}\subseteq\mathcal{O}_{o}\& \tilde{I}|_{H}=I\}$.
Then
\begin{equation}
\label{equ:sharp_jump}
b_{0}-c_{o'}^{I}(\varphi|_{H})\geq b_{1}.
\end{equation}
\end{Theorem}

The sharpness of Theorem \ref{thm:sharp_jump_inequ} is illustrated as follows

\begin{Remark}
Let $(k,n)=(1,2)$ ($H=\{z_{2}=0\}$).
Let $I=(z_{1})_{o'}$, and let $\varphi=\log|z|$.
Then
$c_{o'}^{I}(\varphi|_{H})=2$ and
$c_{o}^{\tilde{I}}(\varphi)=3$
for any $\tilde{I}$ (consider $|z|^{2}$ instead of $|\tilde{I}|^{2}$ and by symmetry),
then
$b_{0}=3$.

Note that
$\log|z|^{4}\geq \log|\tilde{I}h|^{2}+O(1)$.
Then it follows that
$c_{o}^{\tilde{I}h}(\varphi)\geq 4$
for any $\tilde{I}$,
i.e. $b_{1}\geq4-3=1$.
By equality \ref{equ:sharp_jump} and $b_{0}-c_{o'}^{I}(\varphi|_{H})=3-2=1$,
it follows that
$1=b_{0}-c_{o'}^{I}(\varphi|_{H})\geq b_{1}\geq 1$.
Then we obtain the sharpness of Theorem \ref{thm:sharp_jump_inequ}.
\end{Remark}

Let $I=\mathcal{O}_{o'}$, then $\tilde{I}=\mathcal{O}_{o}$.
By Theorem \ref{thm:sharp_jump_inequ},
it follows that
\begin{equation}
\label{equ:sharp_lct}
2c_{o}(\varphi)-c_{o}^{h}(\varphi)\geq c_{o'}(\varphi|_{H}).
\end{equation}

\subsubsection{A slicing result on complex singularity exponents and an application of Theorem \ref{thm:sharp_jump_inequ}}
$\\$

Let $C(V_{k}):=c_{o'}(u|_{V_{k}})$ be a function on
the Grassmannian $G(k, n)$ of $k$- dimensional linear subspaces $V_{k}$ in $\mathbb{C}^n$,
where $o'\in V_{k}$ is the origin.

Stimulated by Siu's slicing theorem on Lelong numbers \cite{siu74},
one can reformulate a
slicing result on complex singularity exponents,
which is implied by the combination of
Berndtsson's log-subharmonicity of Bergman kernels \cite{bern_bergman} and solution of
the openness conjecture:

\begin{Remark}
\label{rem:lct_slice}
There exists $c_{k}\in\mathbb{R}^{+}\cup\{+\infty\}$ such that
$C(V_{k})\equiv c$ almost everywhere in the sense of the unique $U(n)$-invariant measure of
mass 1 on the Grassmannian $G(k, n)$.
Moveover
$c_{k}$ is the upper bound of $c_{o'}(u|_{V_{k}})$ for any $V_{k}$ (details see Lemma \ref{lem:lct_slicing}).
\end{Remark}

When $k=1$, it follows that $c_{o'}(V_{k})=\frac{1}{\nu(u|_{V_{k}},o')}$,
where $\nu(u|_{V_{k}},o')$ is the Lelong number of $u|_{V_{k}}$ at $o'$.
Then Remark \ref{rem:lct_slice} degenerates to Siu's slicing theorem on Lelong numbers \cite{siu74} (see also \cite{demailly-book}) when $k=1$.

Using Theorem \ref{thm:sharp_jump_inequ},
we obtain the following sharp decreasing property of the intervals between consecutive $c_{k}$:
\begin{Corollary}
\label{coro:sharp_lct}
$c_{k}-c_{k-1}\geq c_{k+1}-c_{k}$ holds for any $k\in\{2,\cdots,n-1\}$.
\end{Corollary}

The sharpness of Corollary \ref{coro:sharp_lct} can be seen as follows:

\begin{Remark}
Let $\varphi=\log|z|$,
then $c_{k}=k$.
\end{Remark}

\section{Some preparatory results}\label{sec:preparetory}

In this section,
we recall and present some preparatory results for the proof of the main theorem and applications.

\subsection{Ohsawa-Takegoshi $L^2$ extension theorem}
$\\$

We recall the original form of Ohsawa-Takegoshi $L^2$ extension theorem as follows:

\begin{Theorem}
\label{t:ot_plane}(\cite{ohsawa-takegoshi}, see also \cite{siu96,berndtsson,demailly99}, etc.)
Let $D$ be a bounded pseudo-convex domain in $\mathbb{C}^{n}$.
Let $u$ be a plurisubharmonic function on $D$.
Let $H$ be an $m$-dimensional complex plane in $\mathbb{C}^{n}$.
Then for any holomorphic function on $H\cap D$ satisying
$$\int_{H\cap D}|f|^{2}e^{-2u}d\lambda_{H}<+\infty,$$
there exists a holomorphic function $F$ on $D$ satisfying $F|_{H\cap D}=f$,
and
$$\int_{D}|F|^{2}e^{-2u}d\lambda_{n}\leq C_{D}\int_{H\cap D}|f|^{2}e^{-2u}d\lambda_{H},$$
where $C_{D}$ only depends on the diameter of $D$ and $m$,
and $d\lambda_{H}$ is the Lebesgue measure on $H$.
\end{Theorem}

In \cite{manivel93,demailly99}, Ohsawa-Takegoshi $L^2$ extension theorem has been modified as follows.

\begin{Theorem}
\label{t:ot_plane_MD}(\cite{manivel93,demailly99}, see also \cite{demailly-book,demailly2010})
Let $D$ be a bounded pseudo-convex domain in $\mathbb{C}^{k+1}$.
Let $u$ be a plurisubharmonic function on $D$.
Let $H=\{z_{k+1}=0\}$ be a complex hyperplane in $\mathbb{C}^{k+1}$.
Then for any holomorphic function on $H\cap D$ satisying
$$\int_{H\cap D}|f|^{2}e^{-2u}d\lambda_{H}<+\infty,$$
there exists a holomorphic function $F$ on $D$ satisfying $F|_{H\cap D}=f$,
and
$$\int_{D}|F|^{2}e^{-2u-2a\log|z_{k+1}|}d\lambda_{n}\leq C_{D,a}\int_{H\cap D}|f|^{2}e^{-2u}d\lambda_{H},$$
where $a\in[0,1)$, $C_{D,a}$ only depends on the diameter of $D$ and $a$,
and $d\lambda_{H}$ is the Lebesgue measure on $H$.
\end{Theorem}

For the optimal estimate versions in general settings of Theorem \ref{t:ot_plane_MD} and their applications, the reader is referred to \cite{guan-zhou13aa,guan-zhou15,guan-zhou13ap,GZopen-effect,GZopen-lelong}.

Following the symbols in Theorem \ref{t:ot_plane}, there is a local version of Theorem \ref{t:ot_plane}

\begin{Remark}
\label{rem:ot_plane}(see \cite{ohsawa-takegoshi}, see also \cite{D-K01})
For any germ of holomorphic function $f$ on $o\in H\cap D$ satisfying
$|f|^{2}e^{-2u|_{H}}$ is locally integrable near $o$,
there exists a germ of holomorphic function $F$ on $o\in D$ satisfying $F|_{H\cap D}=f$,
and $|F|^{2}e^{-2u}$ is locally integrable near $o$.
\end{Remark}

\subsection{Strict inequality about jumping numbers}
$\\$

Let $\Omega\ni o$ be a domain in $\mathbb{C}^{n}$,
and let $k=n-1$ and $H=\{z_{n}=0\}$.
Let $I\subseteq O_{o'}$ be a coherent ideal,
and let $v$ be a plurisubharmonic function on $\Delta^{k+1}$ with coordinates $(z_{1},\cdots,z_{k},z_{k+1})$,
where $o'\in H$ is the origin.

Using Theorem \ref{t:ot_plane_MD},
we obtain the following

\begin{Lemma}\label{lem:integ20150130}
Let $I$ be a coherent ideal on $\mathcal{O}_{o'}$.
If $c_{o'}^{I}(v|_{H})=1$,
then for any $N>0$,
$$\sup_{\tilde{I}|_{H}=I}\{c_{o}^{\tilde{I}}(\frac{1}{2}\log(e^{2v}+|z_{k+1}|^{2N}))\}>1$$
holds,
where $\tilde{I}\subseteq\mathcal{O}_{o}$ is a coherent ideal.
\end{Lemma}

\begin{proof}
By H\"{o}lder inequality, it follows that $e^{2(1-\varepsilon)v}|z_{k+1}|^{2\varepsilon N}\leq(1-\varepsilon)e^{2v}+\varepsilon|z_{k+1}|^{2N},$
which implies
\begin{equation}
\label{equ:20150130}
\frac{1}{e^{2v}+|z_{k+1}|^{2N}}\leq e^{-2(1-\varepsilon)v}|z_{k+1}|^{-2\varepsilon N},
\end{equation}
where $\varepsilon\in(0,1)$.

As $c_{o'}^{I}(v|_{H})=1$,
then $|I|^{2}e^{-2(1-\varepsilon)v|_{H}}$ is integrable near $o'$.

By Theorem \ref{t:ot_plane_MD} $($$u\sim (1-\varepsilon)v$, $a\sim\varepsilon N$, $f\sim I$$)$ and choosing $\varepsilon\in(0,\frac{1}{N})$,
it follows that there exists $\tilde{I}$ such that
$|\tilde{I}|^{2}e^{-2(1-\varepsilon)v}|z_{k+1}|^{-2\varepsilon N}$ is locally integrable near $o$.

Using inequality \ref{equ:20150130},
we obtain that $|\tilde{I}|^{2}\frac{1}{e^{2v}+|z_{k+1}|^{2N}}$ is locally integrable near $o$.

Using the strong openness property, we obtain the present lemma.
\end{proof}

Note that
$c_{o}(\frac{1}{2}\log(e^{2v}+|z_{k+1}|^{2N}))=c_{o}(\max\{v,N\log|z_{k+1}|\})$ for any $N>0$
($\Leftarrow$ $e^{2\max\{v,N\log|z_{k+1}|\}}\leq e^{2v}+|z_{k+1}|^{2N}\leq 2e^{2\max\{v,N\log|z_{k+1}|\}}$),
then it follows that

\begin{Corollary}
\label{coro:jump>1}
If $c_{o'}^{I}(v|_{H})=1$,
then
$$\sup_{\tilde{I}|_{H}=I}\{c_{o}^{\tilde{I}}(\max\{2v,\log|z_{k+1}|^{N}\})\}>1$$
for any $N>0$.
\end{Corollary}

After reading the earlier version of the present article, Demailly kindly pointed out that one can obtain an effectiveness result of $\sup_{\tilde{I}|_{H}=I}\{c_{o}^{\tilde{I}}(\frac{1}{2}\log(e^{2v}+|z_{k+1}|^{2N}))\}$ by using the same method as above,
which can deduce the present lemma directly without using the strong openness property.
The details are as follows:

\begin{Lemma}\label{lem:jump_compare0424}
If $c_{o'}^{I}(\varphi|_{H})=b>0$,
$$\sup_{\tilde{I}|_{H}=I}\{c_{o}^{\tilde{I}}(\frac{1}{2}\log(e^{2\varphi}+|z_{k+1}|^{2\frac{N-1}{b}}))\}\geq \frac{bN}{N-1}$$
holds for any $N>1$,
where ideal $\tilde{I}\subseteq\mathcal{O}_{o}$ is coherent.
\end{Lemma}

\begin{proof}

By H\"{o}lder inequality ($ta+(1-t)b\leq a^{t}+b^{1-t}$, where $t\in(0,1)$ and $a>0$, $b>0$), it follows that
$e^{2\varepsilon b\varphi}|z_{k+1}|^{2\varepsilon}\leq \frac{N-1}{N}e^{2\varepsilon \frac{bN}{N-1}\varphi}+\frac{1}{N}|z_{k+1}|^{2\varepsilon N}$
($t\sim \frac{N-1}{N}$, $a\sim e^{2\varepsilon \frac{bN}{N-1}\varphi}$, $b\sim |z_{k+1}|^{2\varepsilon N}$),
which implies
\begin{equation}
\label{equ:20150911}
\frac{1}{e^{2\varepsilon \frac{bN}{N-1}\varphi}+|z_{k+1}|^{2\varepsilon N}}\leq e^{2\varepsilon b\varphi}|z_{k+1}|^{2\varepsilon},
\end{equation}
where $\varepsilon\in(0,1)$.

As $c_{o'}^{I}(\varphi|_{H})=b$,
then $|I|^{2}e^{-2\varepsilon b\varphi|_{H}}$ is integrable near $o'$.
By Theorem \ref{t:ot_plane_MD} $($$\varphi\sim \varepsilon a\varphi$, $a\sim\varepsilon $, $f\sim I$$)$ and choosing $\varepsilon\in(0,\frac{1}{N})$,
it follows that there exists $\tilde{I}$ such that
$|\tilde{I}|^{2}e^{-2\varepsilon b\varphi}|z_{k+1}|^{-2\varepsilon}$ is locally integrable near $o$.
Using inequality \ref{equ:20150911},
we obtain that $|\tilde{I}|^{2}\frac{1}{e^{2\varepsilon \frac{bN}{N-1}\varphi}+|z_{k+1}|^{2\varepsilon N}}$ is locally integrable near $o$.

Note that for any $\varepsilon, N, b$, there exist positive constants $C_{1},C_{2}$ such that
$$C_{1}(e^{2\varphi}+|z_{k+1}|^{2\frac{N-1}{b}})^{\varepsilon\frac{bN}{N-1}}\leq e^{2\varepsilon \frac{bN}{N-1}\varphi}+|z_{k+1}|^{2\varepsilon N}\leq C_{2}(e^{2\varphi}+|z_{k+1}|^{2\frac{N-1}{b}})^{\varepsilon\frac{bN}{N-1}},$$
i.e.,

$(1)$ if $\varepsilon\frac{bN}{N-1}\geq 1$,
then $C_{1}=(\frac{1}{2})^{\varepsilon\frac{bN}{N-1}}$ and $C_{2}=1$;

$(2)$ if $\varepsilon\frac{bN}{N-1}<1$, then $C_{1}=1$ and $C_{2}=2^{\varepsilon\frac{bN}{N-1}}$.

We prove the present lemma.
\end{proof}

Note that
$c_{o}(\frac{1}{2}\log(e^{2\varphi}+|z_{k+1}|^{2\frac{N-1}{b}}))=c_{o}(\max\{\varphi,\frac{N-1}{b}\log|z_{k+1}|\})$ for any $N>0$
($\Leftarrow$ $e^{2\max\{\varphi,\frac{N-1}{b}\log|z_{k+1}|\}}\leq e^{2\varphi}+|z_{k+1}|^{2\frac{N-1}{b}}\leq 2e^{2\max\{\varphi,\frac{N-1}{b}\log|z_{k+1}|\}}$),
then it follows that

\begin{Corollary}
\label{coro:jump_compare20150424}
$\sup_{\tilde{I}|_{H}=I}\{c_{o}^{\tilde{I}}(\max(\varphi,\frac{N-1}{b}\log|z_{k+1}|))\}\geq \frac{bN}{N-1}$
holds
with same symbols and conditions as in Lemma \ref{lem:jump_compare0424}.
\end{Corollary}

\subsection{Hilbert's Nullstellensatz (complex situation) and jumping numbers}
It is well-known that the complex situation of Hilbert's Nullstellensatz is as follows (see (4.22) in \cite{demailly-book})

\begin{Theorem}
\label{thm:hilbert}(see \cite{demailly-book})
For every ideal $I\subset \mathcal{O}_{o}$,
$\mathcal{J}_{V(I),o}=\sqrt{I},$
where $\sqrt{I}$ is the radical of $I$, i.e. the set of germs $f\in\mathcal{O}_{o}$ such that some power $f^{k}$ lies in $I$.
\end{Theorem}

The following lemma can be obtained by the definition of jumping numbers.

\begin{Lemma}
\label{lem:jump_supp}
Let $I\subseteq \mathcal{O}_{o}$ be a coherent ideal, and $u$ be a plurisubharmonic function near $o$.
Then for any $p<(0,c_{o}^{I}(u))$, $(\{z|c_{z}(u)\leq p\},o)\subseteq(V(I),o)$ holds.
\end{Lemma}

Lemma \ref{lem:jump_supp} implies the following

\begin{Remark}
\label{rem:jump_supp}
Let $I\subseteq \mathcal{O}_{o}$ be a coherent ideal, and $u$ be a plurisubharmonic function near $o$.
Let $(A,o)$ be a germ of analytic set such that $c_{z}(u)\leq c_{o}(u)$ for any $z\in (A,o)$
and $\dim((A,o)\setminus(V(I),o))=\dim_{o}A$.
Let $U$ be a neighborhood of $o$ small enough such that $\dim (A\cap U)=\dim_{o}A$.
Then for any $p\in(0,c_{o}(u))$,
$\dim((A\cap U)\setminus\{z|c_{z}(u)\leq p\})=\dim_{o}A$ holds.
Moreover there exists $z_{0}\in ((U\cap A)\setminus V(I))$ such that
$\dim_{z_{0}}A=\dim_{o}A$ and
$z_{0}\not\in(\cup_{p\in(0,c_{o}(u))}\{z|c_{z}(u)\leq p\})$,
which implies $c_{z_{0}}(u)=c_{o}(u)$.
\end{Remark}

\subsection{A useful proposition in \cite{GZopen-lelong} and some generalizations}
$\\$

In \cite{GZopen-lelong},
using Demailly's idea of
equisingular approximations of quasiplurisubharmonic functions (see \cite{demailly2010}, see also \cite{demailly-abel})
and the strong openness property of the multiplier ideal sheaf (see \cite{GZopen-c}),
we have obtained the following proposition:

\begin{Proposition}
\label{Pro:GZ1005}\cite{GZopen-lelong}
Let $D$ be a bounded domain in $\mathbb{C}^{n}$, and the origin $o\in D$.
Let $u$ be a plurisubharmonic function on $D$.
Let $(g_{j})$ be a (finite) local basis of $\mathcal{I}(u)_{o}$.
Then there exists $l>1$
such that
$e^{-2u}-e^{-2\max\{u,\frac{1}{l-1}\log\sum_{j}|g_{j}|\}}$ is integrable
on a small enough neighborhood $V_{o}$ of $o$.
\end{Proposition}

Given a coherent ideal $I\subseteq\mathcal{I}(u)_{o}$
and let $(h_{j})$ be the basis of $I$.
Using $(h_{j})$ instead of $(g_{j})$ in the proof of Proposition \ref{Pro:GZ1005} in \cite{GZopen-lelong},
one can obtain

\begin{Remark}\label{rem:20150130}
For any $I\subset\mathcal{I}(u)_{o}$,
we have $e^{-2u}-e^{-2\max\{u,\frac{1}{l-1}\log|I|\}}<+\infty$.
where $|I|=\sum_{j}|h_{j}|$,
and $l\in(1,c_{o}^{I}(u))$ is the positive number as in Proposition \ref{Pro:GZ1005}.
\end{Remark}

Let $n=k+1$.
It is well-known that
if $\{z|\mathcal{I}(u)_{z}\neq\mathcal{O}_{z}\}\subset\{z_{k+1}=0\}$,
then there exists $N_{0}>0$ large enough
such that $(z_{k+1}^{N_{0}})_{o}\subseteq \mathcal{I}(u)_{o}$.

\begin{Corollary}
\label{Coro:jump=1}
If $c_{o}^{J}(u)\leq 1$
($\Rightarrow$ "$|J|^{2}e^{-2u}$ is not integrable near $o$" by using the strong openness property),
then
$$c_{o}^{J}(\max\{{u},N\log|z_{k+1}|\})\leq1$$
for any $N\geq\frac{1}{l-1}N_{0}$ (independent of $J$),
where $J\subseteq\mathcal{O}_{o}$ is a coherent ideal, $l\in(1,c_{o}^{I}(u))$ and $I:=(z_{k+1}^{N_{0}})_{o}$.
Especially,
if $c_{o}^{J}(u)=1$,
then
$$c_{o}^{J}(\max\{{u},N\log|z_{k+1}|\})=1.$$
\end{Corollary}

\begin{proof}
Using Remark \ref{rem:20150130},
one can obtain that
$e^{-2u}-e^{-2\max\{u,\frac{1}{l-1}N_{0}\log|z_{k+1}|\}}$
is locally integrable near $o$ by letting $I=(z_{k+1}^{N_{0}})_{o}$.

As $|J|^{2}e^{-2u}$ is not locally integrable near $o$ ($\Leftarrow c_{o}^{J}(u)=1$),
then it follows that
$e^{-2\max\{u,\frac{1}{l-1}N_{0}\log|z_{k+1}|\}}$ is not locally integrable near $o$,
which implies
$$c_{o}^{J}(\max\{u,\frac{1}{l-1}N_{0}\log|z_{k+1}|\})\leq 1.$$

Since
$c_{o}^{J}(\max\{u,\frac{1}{l-1}N_{0}\log|z_{k+1}|\})\geq c_{o}^{J}(u)=1$ ($\Leftarrow \max\{u,\frac{1}{l-1}N_{0}\log|z_{k+1}|\}\geq u$),
then it follows that
$c_{o}^{J}(\max\{u,\frac{1}{l-1}N_{0}\log|z_{k+1}|\})=1.$

As $N\log|z_{k+1}|\leq \frac{1}{l-1}N_{0}\log|z_{k+1}|$,
then it follows that
$u\leq \max\{{u},N\log|z_{k+1}|\} \leq \max\{u,\frac{1}{l-1}N_{0}\log|z_{k+1}|\},$
which implies
$$c_{o}^{J}(u)\leq c_{o}(\max\{u,N\log|z_{k+1}|\}) \leq c_{o}^{J}(\max\{u,\frac{1}{l-1}N_{0}\log|z_{k+1}|\}).$$

Note that $c_{o}^{J}(u)=c_{o}^{J}(\max\{u,\frac{1}{l-1}N_{0}\log|z_{k+1}|\})=1$,
then we obtain the present corollary.
\end{proof}

Let $I\subseteq\mathcal{O}_{o}$.
Let $J$ be an coherent ideal satisfying $IJ\subseteq\mathcal{I}(c^{I}_{o}(\varphi)\varphi)_{o}$
($\Leftrightarrow IJ\subseteq\mathcal{I}_{o}(c_{o}^{I}(\varphi)\varphi)$).
Let $\{f_{j}\}_{j=1,2,\cdots,s}$ be a local basis of $J_{o}$.
Denoted by $|J|:=\sum_{i=1}^{s}|f_{i}|$.
Let $\tilde{\varphi}_{l}:=\max\{c^{I}_{o}(\varphi)\varphi,\frac{1}{l-1}\log|J|\}$.

If $\psi_{1}-C\leq\psi\leq\psi_{1}+C$,
then
\begin{equation}
\label{equ:20150524a}
\begin{split}
\max\{\varphi,\psi_{1}\}-C
&\leq \max\{\varphi,\psi_{1}-C\}
\\&\leq \max\{\varphi,\psi\}
\\&\leq\max\{\varphi,\psi_{1}+C\}\leq \max\{\varphi,\psi_{1}\}+C,
\end{split}
\end{equation}
where $\varphi$, $\psi_{1}$ and $\psi$ are plurisubharmonic functions.

By inequality \ref{equ:20150524a}
($\varphi\sim c^{I}_{o}(\varphi)\varphi$,
$\psi\sim\frac{1}{l-1}\log|J|$),
it follows that
$$c_{o}^{I}(\max\{c^{I}_{o}(\varphi)\varphi,\frac{1}{l-1}\log|J|\})$$
is well-defined for any basis of $J$.
The proof of Proposition 2.1 in \cite{GZopen-lelong} also implies the following

\begin{Remark}\label{rem:GZ1005}
For any $l\in(1,\frac{c_{o}^{IJ}(\varphi)}{c_{o}^{I}(\varphi)}]$,
we have $c_{o}^{I}(\tilde{\varphi}_{l})=1.$
\end{Remark}

\begin{proof}
For the convenience of the readers,
we recall the proof in \cite{GZopen-lelong} with subtle modifications as follows:

It is clear that there exists a small enough neighborhood $V_{1}\ni o$
such that

\begin{equation}
\label{equ:20150524b}
\int_{V_{1}}|IJ|^{2}e^{-2c^{I}_{o}(\varphi)\varphi}<\infty.
\end{equation}

Given any real number $l\in(1,\frac{c_{o}^{IJ}(\varphi)}{c_{o}^{I}(\varphi)})$,
by the strong openness property,
there exists a small neighborhood $V_{2}$ of $o$ such that
\begin{equation}
\label{equ:20141005a}
\int_{V_{2}}|IJ|^{2}e^{-2lc_{o}^{I}(\varphi)\varphi}<\infty.
\end{equation}

Then
\begin{equation}
\label{equ:20141005b}
\begin{split}
\int_{V_{2}}|I|^{2}(e^{-2c_{o}^{I}(\varphi)\varphi}-e^{-2\tilde{\varphi}_{l}})
&\leq\int_{\{\varphi<\frac{1}{(l-1)c_{o}^{I}(\varphi)}\log|J|\}\cap V_{2}}|I|^{2}e^{-2c_{o}^{I}(\varphi)\varphi}
\\&=\int_{\{\varphi<\frac{1}{(l-1)c_{o}^{I}(\varphi)}\log|J|\}\cap V_{2}}|I|^{2}e^{2(l-1)c_{o}^{I}(\varphi)\varphi-2lc^{I}_{o}(\varphi)\varphi}
\\&
\leq\int_{\{\varphi<\frac{1}{(l-1)c_{o}^{I}(\varphi)}\log|J|\}\cap V_{2}}|I|^{2}e^{2\log|J|-2lc_{o}^{I}(\varphi)\varphi}
\\&\leq\int_{V_{2}}|I|^{2}|J|^{2}e^{-2lc_{o}^{I}(\varphi)\varphi}<+\infty,
\end{split}
\end{equation}
where the last inequality follows from inequality \ref{equ:20141005a}.

As
$|I|^{2}(e^{-2c_{o}^{I}(\varphi)\varphi}-e^{-2\tilde{\varphi}_{l}})$ is integrable near $o$,
and $|I|^{2}e^{-2c_{o}^{I}(\varphi)\varphi}$ is not integrable near $o$,
it follows that
$|I|^{2}e^{-2\tilde{\varphi}_{l}}$ is not integrable near $o$,
which implies $c_{o}^{I}(\tilde{\varphi}_{l})\leq1$.

As $c_{o}^{I}(\varphi)\varphi\leq \tilde{\varphi}_{l}$,
it follows that
$$|I|^{2}e^{-2c c_{o}^{I}(\varphi)\varphi}\geq |I|^{2}e^{-2c\tilde{\varphi}_{l}}$$
for any $c>0$.
When $c\in(0,1)$,
by the definition of jumping numbers,
it follows that
$|I|^{2}e^{-2c c_{o}^{I}(\varphi)\varphi}$ is locally integrable near $o$,
which implies
$|I|^{2}e^{-2c\tilde{\varphi}_{l}}$ is locally integrable near $o$,
i.e.,
$c_{o}^{I}(\tilde{\varphi}_{l})\geq1.$
Then we have $c_{o}^{I}(\tilde{\varphi}_{l})=1.$
\end{proof}

We recall a consequence of Proposition \ref{Pro:GZ1005} as follows
\begin{Remark}
\label{rem:GZ1005}(\cite{GZopen-lelong})
$\tilde{u}$ (as in Proposition \ref{Pro:GZ1005}) has the following properties

$(1)$ for any $z\in(\{z|c_{z}(u)\leq1\},o)=(V(\mathcal{I}(u)),o)$, inequality $c_{z}(u)\leq c_{z}(\tilde{u})\leq 1$ holds;

$(2)$ if $c_{z_{0}}(u)=1$, then $c_{z_{0}}(\tilde{u})=1$, where $z_{0}\in(\{z|c_{z}(u)\leq1\},o)$.
\end{Remark}

Let $I\subseteq\mathcal{O}_{o}$ be a coherent ideal,
and let $u$ be a plurisubharmonic function near $o$.
We present the following consequence of Remark \ref{rem:GZ1005} about the integrability of the ideals related to weight of jumping number one.

\begin{Proposition}
\label{thm:jump_supp}
Let $J\subseteq\mathcal{O}_{o}$ be a coherent ideal.
Assume that $c_{o}^{I}(u)=1$.
If $(V(\mathcal{I}(u)),o)\subseteq(V(J),o)$,
then $|I|^{2}|J|^{2\varepsilon}e^{-2u}$ is locally integrable near $o$ for any $\varepsilon>0$.
\end{Proposition}
After the present article has been written,
Demailly kindly shared his manuscript \cite{demailly2015} with the first author of the present article,
which includes Proposition \ref{thm:jump_supp} (Lemma (4.2) in \cite{demailly2015}) as a consequence of the strong openness property of the multiplier ideal (see \cite{GZopen-c}).

\begin{proof}(proof of Proposition \ref{thm:jump_supp})
Let $J_{0}\subseteq\mathcal{O}_{o}$ be a coherent ideal satisfying
$(V(J_{0}),o)\supseteq(V(\mathcal{I}(u)),o).$
By Theorem \ref{thm:hilbert} $(I\sim \mathcal{I}(u)_{o})$,
it follows that there exists large enough positive integer $N$
such that $J^{N}_{0}\subseteq\mathcal{I}(u)_{o}$.

By Remark \ref{rem:GZ1005},
it follows that exist $p_{0}>0$ large enough
such that
$e^{-2u}-e^{-2\max\{u,p_{0}\log|J_{0}|\}}$
is locally integrable near $o$.

It suffices to prove that
$|I|^{2}|J_{0}|^{2\varepsilon}e^{-2\max\{u,p_{0}\log|J_{0}|\}}$ is locally integrable near $o$ for small enough $\varepsilon>0$.
We prove the above statement by contradiction:
If not, then there exists $\varepsilon_{0}>0$,
such that $|I|^{2}|J_{0}|^{2\varepsilon_{0}}e^{-2\max\{u,p_{0}\log|J_{0}|\}}$ is not locally integrable near $o$.
Note that $\varepsilon_{0}\log|J_{0}|\leq \frac{\varepsilon_{0}}{p_{0}}\max\{u,p_{0}\log|J_{0}|\}$,
then it follows that
$|I|^{2}e^{-2(1-\frac{\varepsilon_{0}}{p_{0}})\max\{u,p_{0}\log|J_{0}|\}}$ is not locally integrable near $o$.
Note that $u\leq\max\{u,p_{0}\log|J_{0}|\}$,
then it follows that $|I|^{2}e^{-2(1-\frac{\varepsilon_{0}}{p_{0}})u}$ is not locally integrable near $o$,
which contradicts $c_{o}^{I}(u)=1$.
Then we prove Proposition \ref{thm:jump_supp}.
\end{proof}
Let $I=\mathcal{O}_{o}$, $\varepsilon=1$.
Using Proposition \ref{thm:jump_supp},
we obtain the following result.

\begin{Corollary}
\label{coro:jump_supp}Let $J\subseteq\mathcal{O}_{o}$ be a coherent ideal.
Assume that $c_{o}(u)=1$.
Then the following two statements are equivalent\\
$(1)$ $(V(J),o)\supseteq(V(\mathcal{I}(u)),o)$;\\
$(2)$ $|J|^{2}e^{-2u}$ is locally integrable near $o$, i.e. $J\subseteq \mathcal{I}(u)_{o}$.
\end{Corollary}

\subsection{Measures along the fibres}
$\\$

Let $X:=\{z_{k+1}=\cdots=z_{n}=0\}$.
Consider a map $q$ from $\mathbb{C}^{n}\setminus X$ to $\mathbb{C}\mathbb{P}^{n-k-1}$:
$q(z_{1},\cdots,z_{n})=(z_{k+1}:\cdots:z_{n}).$

Let $Y$ be an analytic set in $\mathbb{B}^{n}$
whose complex dimension is smaller than $n-k$.
By the same proof as that of Lemma 2.8 in \cite{GZopen-lelong}
(the methods can be found in \cite{Grau-Rem84} and \cite{Chirka89}),
one can obtain that

\begin{Lemma}
\label{l:intersec_fibre}
For almost all $(z_{k+1}:\cdots:z_{n})$,
the complex dimension of $q^{-1}(z_{k+1}:\cdots:z_{n})\cap Y$ is zero,
i.e.,
$(\overline{q^{-1}(z_{k+1}:\cdots:z_{n})}\cap Y,o)=(X\cap Y,o)$.
\end{Lemma}

\begin{proof}
Note that the  $2(n-k)-2$ dimensional Hausdorff measure on $\mathbb{C}\mathbb{P}^{n-k-1}$ is positive,
and $2(n-k)$ dimensional Hausdorff measure of $Y$ is zero,
then it follows that for almost all $(z_{k+1}:\cdots:z_{n})$,
the $2$ dimensional Hausdorff measure of $q^{-1}(z_{k+1}:\cdots:z_{n})\cap Y$ is zero,
i.e.,
the complex dimension of $q^{-1}(z_{k+1}:\cdots:z_{n})\cap Y$ is zero.
Then we obtain the present lemma.
\end{proof}

\subsection{Proof of Proposition \ref{thm:add_prod_cse}}\label{subsec:add_prod}
$\\$

By the definition of jumping numbers,
it follows that for any $\varepsilon>0$,
there exists a neighborhood $U_{\varepsilon}$ of $o$ and $C_{\varepsilon}>0$
such that
$$r^{-2(c_{o}^{I}(u)-\varepsilon)}\int_{\Delta^{n}}\mathbb{I}_{\{u<\log r\}\cap U_{\varepsilon}}|I|^{2}d\lambda_{n}<C_{\varepsilon}$$
holds for any $r>0$,
which implies
\begin{equation}
\label{equ:DK_GZ0322a}
\liminf_{r\to0^{+}}\frac{\log(\int_{\Delta^{n}}\mathbb{I}_{\{u<\log r\}\cap U_{\varepsilon}}|I|^{2}d\lambda_{n})}{2\log r}\geq c_{o}^{I}(u)-\varepsilon.
\end{equation}

We recall our solution of a conjecture posed by Demailly-Kollar \cite{D-K01} (which means that
$\liminf_{r\to0^{+}}(-r^{2c_{o}^{I}(u)}\int_{\Delta^{n}}\mathbb{I}_{\{u<\log r\}\cap U}d\lambda_{n})>0,$
holds)
as follows
\begin{Theorem}
\label{thm:DK_GZ}\cite{GZopen-b,GZopen-effect}
Let $u$ be a plurisubharmonic function on $\Delta^{n}\subset\mathbb{C}^{n}$
and $I$ be a coherent ideal in $\mathcal{O}_{o}$.
Then for any neighborhood $U$ of $o$,
there exists $C_{\varepsilon}>0$ such that
$$(-r^{2c_{o}^{I}(u)}\int_{\Delta^{n}}\mathbb{I}_{\{u<\log r\}\cap U}|I|^{2}d\lambda_{n})>C_{\varepsilon}.$$
\end{Theorem}

By Theorem \ref{thm:DK_GZ},
it follows that for any neighborhood $U$ of $o$,
\begin{equation}
\label{equ:DK_GZ0322b}
\limsup_{r\to0^{+}}\frac{\log(\int_{\Delta^{n}}\mathbb{I}_{\{u<\log r\}\cap U}|I|^{2}d\lambda_{n})}{2\log r}\leq c_{o}^{I}(u)
\end{equation}
holds.

As $\{\max\{u\circ\pi_{1},v\circ\pi_{2}\}<\log r\}\cap (\pi_{1}^{-1}(U)\cap\pi_{2}^{-1}(V))=(\{u<\log r\}\cap U)\times(\{v<\log r\}\cap V)$,
then it follows that
\begin{equation}
\label{equ:DK_GZ0322c}
\begin{split}
\int_{\Delta^{n}\times\Delta^{n}}&\mathbb{I}_{\{\max\{u\circ\pi_{1},v\circ\pi_{2}\}<\log r\}\cap (\pi_{1}^{-1}(U)\cap\pi_{2}^{-1}(V))}(|\pi_{1}^{*}I|\times|\pi_{2}^{*}J|)^{2}d\lambda_{2n})
\\=&\int_{\Delta^{n}}\mathbb{I}_{\{u<\log r\}\cap U}|I|^{2}d\lambda_{n}\times
\int_{\Delta^{n}}\mathbb{I}_{\{v<\log r\}\cap V}|J|^{2}d\lambda_{n},
\end{split}
\end{equation}
where $U$ and $V$ are neighborhoods of $o\in\mathbb{C}^{n}$, $I$ and $J$ are coherent ideals in $\mathcal{O}_{o}$.

By inequality \ref{equ:DK_GZ0322a},
it follows that for any $\varepsilon>0$,
there exist neighborhoods $U_{\varepsilon}$ and $V_{\varepsilon}$ of $o$
such that
$$\liminf_{r\to0^{+}}\frac{\log\int_{\Delta^{n}}\mathbb{I}_{\{u<\log r\}\cap U_{\varepsilon}}|I|^{2}d\lambda_{n}}{2\log r}\geq c_{o}^{I}(u)-\varepsilon$$
and
$$\liminf_{r\to0^{+}}\frac{\log\int_{\Delta^{n}}\mathbb{I}_{\{v<\log r\}\cap V_{\varepsilon}}|J|^{2}d\lambda_{n}}{2\log r}\geq c_{o}^{J}(v)-\varepsilon.$$
By inequality \ref{equ:DK_GZ0322b},
it follows that
\begin{equation}
\label{equ:DK_GZ0322d}
\begin{split}
&
c_{o}^{I\times J}(\max\{u\circ\pi_{1},v\circ\pi_{2}\})
\\=&\limsup_{r\to0^{+}}\frac{\log\int_{\Delta^{n}\times\Delta^{n}}\mathbb{I}_{\{\max\{u\circ\pi_{1},v\circ\pi_{2}\}<\log r\}\cap (\pi_{1}^{-1}(U_{\varepsilon})\cap\pi_{2}^{-1}(V_{\varepsilon}))}(|\pi_{1}^{*}I|\times|\pi_{2}^{*}J|)^{2}d\lambda_{2n}}
{2\log r}
\\\geq&
\liminf_{r\to0^{+}}\frac{\log\int_{\Delta^{n}\times\Delta^{n}}\mathbb{I}_{\{\max\{u\circ\pi_{1},v\circ\pi_{2}\}<\log r\}\cap (\pi_{1}^{-1}(U_{\varepsilon})\cap\pi_{2}^{-1}(V_{\varepsilon}))}(|\pi_{1}^{*}I|\times|\pi_{2}^{*}J|)^{2}d\lambda_{2n}}
{2\log r}
\\=&
\liminf_{r\to0^{+}}\frac{\log\int_{\Delta^{n}}\mathbb{I}_{\{u<\log r\}\cap U_{\varepsilon}}|I|^{2}d\lambda_{n}}{2\log r}+
\liminf_{r\to0^{+}}\frac{\log\int_{\Delta^{n}}\mathbb{I}_{\{v<\log r\}\cap V_{\varepsilon}}|J|^{2}d\lambda_{n}}{2\log r}
\\\geq&
(c_{o}^{I}(u)-\varepsilon)+(c_{o}^{J}(v)-\varepsilon).
\end{split}
\end{equation}

By inequality \ref{equ:DK_GZ0322a},
it follows that for any $\varepsilon>0$,
there exist neighborhoods $U'_{\varepsilon}$ and $V'_{\varepsilon}$ of $o$
such that
\begin{equation}
\label{equ:DK_GZ0322e}
\begin{split}
\liminf_{r\to0^{+}}&\frac{\log\int_{\Delta^{n}\times\Delta^{n}}\mathbb{I}_{\{\max\{u\circ\pi_{1},v\circ\pi_{2}\}<\log r\}\cap (\pi_{1}^{-1}(U'_{\varepsilon})\cap\pi_{2}^{-1}(V'_{\varepsilon}))}
(|\pi_{1}^{*}I|\times|\pi_{2}^{*}J|)^{2}d\lambda_{2n}}{2\log r}
\\&\geq c_{o}^{I\times J}(\max\{u\circ\pi_{1},v\circ\pi_{2}\})-\varepsilon.
\end{split}
\end{equation}
By inequality \ref{equ:DK_GZ0322b},
it follows that
\begin{equation}
\label{equ:DK_GZ0322f}
\begin{split}
&c_{o}^{I}(u)+c_{o}^{J}(v)
\\\geq&\limsup_{r\to0^{+}}\frac{\log \int_{\Delta^{n}}\mathbb{I}_{\{u<\log r\}\cap U'_{\varepsilon}}|I|^{2}d\lambda_{n}}{2\log r}+
\limsup_{r\to0^{+}}\frac{\log \int_{\Delta^{n}}\mathbb{I}_{\{v<\log r\}\cap V'_{\varepsilon}}|J|^{2}d\lambda_{n}}{2\log r}
\\=&\limsup_{r\to0^{+}}\frac{\log \int_{\Delta^{n}\times\Delta^{n}}\mathbb{I}_{\{\max\{u\circ\pi_{1},v\circ\pi_{2}\}<\log r\}\cap (\pi_{1}^{-1}(U'_{\varepsilon})\cap\pi_{2}^{-1}(V'_{\varepsilon}))}
(|\pi_{1}^{*}I|\times|\pi_{2}^{*}J|)^{2}d\lambda_{2n}}{2\log r}
\\\geq&
\liminf_{r\to0^{+}}\frac{\log\int_{\Delta^{n}\times\Delta^{n}}\mathbb{I}_{\{\max\{u\circ\pi_{1},v\circ\pi_{2}\}<\log r\}\cap (\pi_{1}^{-1}(U'_{\varepsilon})\cap\pi_{2}^{-1}(V'_{\varepsilon}))}
(|\pi_{1}^{*}I|\times|\pi_{2}^{*}J|)^{2}d\lambda_{2n}}{2\log r}
\\\geq&
c_{o}^{I\times J}(\max\{u\circ\pi_{1},v\circ\pi_{2}\})-\varepsilon.
\end{split}
\end{equation}

Letting $\varepsilon \rightarrow 0$, using inequality \ref{equ:DK_GZ0322d} and inequality \ref{equ:DK_GZ0322f},
we obtain Proposition \ref{thm:add_prod_cse}.

\subsection{Slicing result on complex singularity exponent and subadditivity theorem on jumping numbers}
$\\$

Let $v$ be a plurisubharmonic function on $\Delta^{n}$.
Let $\mathcal{H}_{2v}(\Delta^{n})$ be the Hilbert space of the holomorphic function $f$ on $\Delta^{n}$ satisfying
(the norm) $(\int_{\Delta^{n}}|f|^{2}e^{-2v})^{1/2}<+\infty$.
Let $K_{\Delta^{n},2v}$ be the Bergman kernel associated with $\mathcal{H}_{2v}(\Delta^{n})$.

It is easy to see that $\int_{\Delta^{n}_{r}}e^{-2v}=+\infty$ (for any $r>0$) if and only if
$K_{\Delta^{n},2v}(o)=0$, where $o$ is the origin in $\mathbb{C}^{n}$.

By definition of $c_{o}(v)$,
it follows that $\int_{\Delta^{n}_{r}}e^{-2v}=+\infty$ (for any $r>0$) implies $c_{o}(v)\leq 1$;
by Berndtsson's solution of
the openness conjecture,
it follows that $c_{o}(v)\leq 1$ implies $\int_{\Delta^{n}_{r}}e^{-2v}=+\infty$ (for any $r>0$).
Then one can obtain

\begin{Lemma}
\label{lem:bergman_zero}
$c_{o}(v)\leq 1$ if and only if $K_{\Delta^{n},2v}(o)=0$.
\end{Lemma}

Let $p:\Delta^{n}\times\Delta^{m}\to\Delta^{m}$ be the projection
satisfying $p(z_{1},\cdots,z_{n},w_{1},\cdots,w_{m})=(w_{1},\cdots,w_{m})$,
where $(z_{1},\cdots,z_{n})$ and $(w_{1},\cdots,w_{m})$ are coordinates on
$\mathbb{C}^{n}$ and $\mathbb{C}^{m}$.

Let $u$ be a plurisubharmonic function on $\Delta^{n}\times\Delta^{m}$.
Let $K_{2cu}$ be the fiberwise Bergman kernel on $\Delta^{n}\times\Delta^{m}$
such that
$K_{2cu}|_{p^{-1}(w)}$ is the Bergman kernel associated with the Hilbert space
$\mathcal{H}_{2cu|_{p^{-1}(w)}}(p^{-1}(w))$ (see \cite{bern_bergman}).

Berndtsson's important result on log-subharmonicity of the Bergman kernels \cite{bern_bergman}
tells us that $\log K_{2cu}$ is plurisubharmonic.
Combining with Lemma \ref{lem:bergman_zero},
one can obtain
\begin{Lemma}
\label{lem:lct_slicing}
For any $a>0$, $\{w|c_{(o,w)'}(u|_{p^{-1}(w)})\leq a\}$ is a complete pluripolar set on $\Delta^{m}$,
which implies that $c_{(o,w)'}(u|_{p^{-1}(w)})$ are the same (denoted by $C$) for almost every $w$ in the sense of Lebesgue measure on $\mathbb{C}^{m}$,
$(o,w)'\in p^{-1}(w)$.
Moveover $C=\sup_{w\in\Delta^{m}}\{c_{(o,w)'}(u|_{p^{-1}(w)})\}$.
\end{Lemma}

\begin{proof}
By Lemma \ref{lem:bergman_zero} $(v=au|_{p^{-1}(w)})$, it follows that
$$\{w|c_{(o,w)'}(u|_{p^{-1}(w)})\leq a\}=\{w|\log K_{2au}(o,w)=-\infty\},$$
which implies
$\{w|c_{(o,w)'}(u|_{p^{-1}(w)})\leq a\}$ is a complete pluripolar set on $\Delta^{m}$.

Note that the Lebesgue measure of pluripolar set is $0$ or $\pi^{m}$.
It follows that $c_{(o,w)'}(u|_{p^{-1}(w)})$ are the same (denoted by $C$) for almost every $w$.

We prove "moreover" part by contradiction:
if not,
then it follows that there exists $w$ satisfying $c_{(o,w)'}(u|_{p^{-1}(w)})>C$,
which implies
$$\log K_{2Cu}(o,w)>-\infty.$$

As $\log K_{2Cu}(o,w)$ is plurisubharmonic,
then it follows that
there exists a neighborhood $U$ of $w$
such that
$$\log K_{2Cu}(o,\cdot)>-\infty$$
for almost all point in $U$.

Using Berndtsson's solution of the openness conjecture,
one can obtain
$$c_{(o,\cdot)'}(u|_{p^{-1}(\cdot)}(o,\cdot))>C$$
holds for almost all point in $U$,
which contradicts "$c_{(o,w)'}(u|_{p^{-1}(w)})=C$ for almost all $w\in\Delta^{m}$".
\end{proof}

By the strong openness property,
one can also obtain an analogue of the restriction formula for multiplier ideal,
$$c_{o}^{I}(u\circ f)\leq\sup\{c^{\tilde{I}}_{f(o)}(u)|f^{*}\tilde{I}=I\ \& \ \tilde{I}\subseteq\mathcal{O}_{f(o)}\},$$
which
is equivalent to
the comparison property on the multiplier ideals: $\mathcal{I}(u\circ f)\subseteq f^{*}\mathcal{I}(u)$
(see \cite{DEL00}, see also (14.3) in \cite{demailly2010}),
where $f$ is a holomorphic map.\\

In \cite{DEL00} (see also Theorem (14.2) in \cite{demailly2010}),
the following subadditivity theorem on jumping numbers has been presented

\begin{Theorem}
\label{thm:add_multi}Subadditivity Theorem
$\\$
$(a)$
$\pi_{i}:=\Omega_{1}\times\Omega_{2}\to\Omega_{i}$
$i=1,2$ the projections,
and let $u_{i}$ be a plurisubharmonic function on $\Omega_{i}$.
Then
$$\mathcal{I}(u_{1}\circ \pi_{1}+u_{2}\circ \pi_{2})=\pi_{1}^{*}(\mathcal{I}(u_{1}))\cdot\pi_{2}^{*}(\mathcal{I}(u_{2})).$$

$\\$
$(b)$ Let $\Omega$ be a domain and let $u$ and $v$ be two plurisubharmonic functions on $\Omega$.
Then
$$\mathcal{I}(u+v)\subseteq\mathcal{I}(u)\cdot\mathcal{I}(v).$$
\end{Theorem}

By the strong openness property,
it follows that Theorem \ref{thm:add_multi} is equivalent to the following result:

\begin{Theorem}
\label{thm:add_jump}
$\\$
$(a)$ Let $\Omega_{1}\ni o_{1}$ and $\Omega_{2}\ni o_{2}$ be two domains,
$\pi_{i}:=\Omega_{1}\times\Omega_{2}\to\Omega_{i}$
$i=1,2$ the projections,
and let $u_{i}$ be a plurisubharmonic function on $\Omega_{i}$.
Then
\begin{equation}
\label{equ:add_jump_equ}
c_{o_{1}\times o_{2}}^{\tilde{I}}(u_{1}\circ \pi_{1}+u_{2}\circ \pi_{2})=\sup\{\min\{c_{o_{1}}^{J_{1}}(u_{1}),c_{o_{2}}^{J_{2}}(u_{2})\}|J_{1}\cdot J_{2}\supseteq \tilde{I}\},
\end{equation}
where $\tilde{I}$ is a coherent ideal in $\mathcal{O}_{o_{1}\times o_{2}}$, $J_{1}$ and $J_{2}$ are coherent ideals in $\mathcal{O}_{o_{1}}$ and $\mathcal{O}_{o_{2}}$ respectively.
$\\$
$(b)$ Let $\Omega$ be a domain, let $u$ and $v$ be two plurisubharmonic functions on $\Omega\ni o$.
Then
\begin{equation}
\label{equ:add_jump_inequ}
c_{o}^{I}(u+v)\leq\sup\{\min\{c_{o}^{I_{1}}(u),c_{o}^{I_{2}}(v)\}|I_{1}\cdot I_{2}\supseteq I\}
\end{equation}
where $I_{1}$ and $I_{2}$ are coherent ideals in $\mathcal{O}_{o}$.
\end{Theorem}

Let $\Omega_{i}\subset\mathbb{C}^{n}$ and containing the origin $o\in\mathbb{C}^{n}$ for any $i\in\{1,2\}$.
Let $\Delta$ be the diagonal of $\mathbb{C}^{n}\times\mathbb{C}^{n}$.
It is well-known that

\begin{Remark}\label{lem:linear}
Let $A_{1}$ and $A_{2}$ be two varieties on $\Omega_{1}$ and $\Omega_{2}$ respectively through $o$.
Assume that $A_{1}$ and $A_{2}$ are both regular at $o$.
Then $\dim (T_{A_{1},o}\cap T_{A_{2},o})=\dim(T_{A_{1}\times A_{2},(o,o)}\cap T_{\Delta,(o,o)})$.
\end{Remark}

\subsection{Applications of the slicing result on complex singularity exponent}

Let $(z_{1},\cdots,z_{k})$ be the coordinates of
$\mathbb{B}^{k-l}\times\mathbb{B}^{l}\subseteq\mathbb{C}^{k}$,
and let
$p:\mathbb{B}^{k-l}\times\mathbb{B}^{l}\to\mathbb{B}^{k-l}$.
Let $H_{1}:=\{z_{k-l+1}=\cdots=z_{k}=0\}$.

We present a corollary of Lemma \ref{lem:lct_slicing} as follows

\begin{Corollary}
\label{coro:lct_slice_open}
Let $u$ be a plurisubharmonic function on $\mathbb{B}^{k-l}\times\mathbb{B}^{l}$.
Assume that $c_{z}(u)\leq 1$ for any $z\in H_{1}$ and $c_{o}(u)=1$,
where $o$ is the origin in $\mathbb{B}^{k-l}\times\mathbb{B}^{l}$.
Then for almost every
$a=(a_{1},\cdots,a_{k-l})\in\mathbb{B}^{k-l}$
with respect to the Lebesgue measure on $\mathbb{C}^{k-l}$, $c_{z_{a}'}(u|_{L_{a}})=1$ holds,
where $L_{a}=\{z_{1}=a_{1},\cdots,z_{k-l}=a_{k-l}\}$,
and $z_{a}'\in L_{a}\cap H_{1}$ emphasizes that $c_{z_{a}'}(u|_{L_{a}})$ is computed on the submanifold $L_{a}$.
\end{Corollary}

\begin{proof}
By Lemma \ref{lem:lct_slicing} and $c_{o}(u)=1$,
it follows that $c\geq 1$ (consider the integrability of $e^{-2pu}$ near $o$, where $p<1$ near $1$, and by contradiction).

By $c_{z}(u)\leq 1$ for any $z\in H_{1}$ and Proposition \ref{prop:DK2000},
it follows that $c_{z_{a}'}(u|_{L_{a}})\leq c_{z_{a}}(u)\leq1$ for any $z_{a}'\in L_{a}\cap H_{1}$.
Combining $c\geq 1$, we obtain Corollary \ref{coro:lct_slice_open}.
\end{proof}

The following remark is the singular version of Corollary \ref{coro:lct_slice_open}:

\begin{Remark}
\label{rem:lct_slice_sing}
Let $A_{3}$ be a reduced irreducible analytic subvariety on $\mathbb{B}^{k-l}\times\mathbb{B}^{l}$ through $o$
satisfying $\dim_{o}A_{3}=k-l$
such that

$(1)$ for any $a=(a_{1},\cdots,a_{k-l})\in\mathbb{B}^{k-l}$,
$A_{3}\cap L_{a}\neq\emptyset$, where $L_{a}=\{z_{1}=a_{1},\cdots,z_{k-l}=a_{k-l}\}$;

$(2)$ there exists analytic subset $A_{4}\subseteq\mathbb{B}^{k-l}$
such that any $z\in (A_{3}\cap p^{-1}(\mathbb{B}^{k-l}\setminus A_{4}))$
is the regular point in $A_{3}$ and the noncritical point of $p|_{A_{3,reg}}$.

Let $u$ be a plurisubharmonic function on $\mathbb{B}^{k-l}\times\mathbb{B}^{l}$.
Assume that $c_{z}(u)\leq 1$ for any $z\in A_{3}$ and $c_{o}(u)=1$.
Then for almost every
$a=(a_{1},\cdots,a_{k-l})\in\mathbb{B}^{k-l}$
with respect to the Lebesgue measure on $\mathbb{C}^{k-l}$,
there exists $z_{a}\in A_{3}\cap L_{a}$
such that
equality $c_{z_{a}'}(u|_{L_{a}})=1$.
\end{Remark}

\begin{proof}
By Lemma \ref{lem:lct_slicing},
it follows that there exists $c\in(0,\infty]$.
such that $c_{z'}(u|_{L_{p(z)}})=c$
for almost every $z\in (A_{3}\cap p^{-1}(\mathbb{B}^{k-l}\setminus A_{4}))$
with respect to the Lebesgue measure
in $(A_{3}\cap p^{-1}(\mathbb{B}^{k-l}\setminus A_{4}))$.
By $c_{o}(u)=1$,
it follows that $c\geq 1$
(consider the integrability of $e^{-2pu}$ near $o$, where $p<1$ near $1$,
and by contradiction).

By Proposition \ref{prop:DK2000},
it follows that $1\leq c\leq c_{z}(u|_{L_{p(z)}})\leq c_{z}(u)\leq 1$ holds,
for almost every $z\in (A_{3}\cap p^{-1}(\mathbb{B}^{k-l}\setminus A_{4}))$.
Then one can find $z_{3}\in (A_{3}\cap p^{-1}(\mathbb{B}^{k-l}\setminus A_{4}))$
such that $c_{z_{3}}(u)=c_{z_{3}'}(u|_{L_{p(z_{3})}})=1$.

By Corollary \ref{coro:lct_slice_open} $(o\sim z_{3})$,
it follows that
for almost every
$a=(a_{1},\cdots,a_{k-l})\in(\mathbb{B}^{k-l}\setminus A_{4})$
with respect to the Lebesgue measure on $\mathbb{C}^{k-l}$,
there exists $z_{a}\in A_{3}\cap L_{a}$
such that
equality $c_{z_{a}'}(u|_{L_{a}})=1$.
As the Lebesgues measure of $A_{4}$ on $\mathbb{C}^{k-l}$ is zero,
then we obtain Remark \ref{rem:lct_slice_sing}.
\end{proof}

\section{Proof of Theorem \ref{t:jump_equality} (main theorem)}\label{sec:proof_main}

It suffices to consider the case $c_{o'}^{I}(u|_{H})=1$ (consider $c_{o'}^{I}(u|_{H})u$ instead of $u$).

By $\sup\{c_{o}^{\tilde{I}}(u)|\tilde{I}\subseteq\mathcal{O}_{o}\ \& \ \tilde{I}|_{H}=I\}=1$ ($\Rightarrow$ $|\tilde{I}|^{2}e^{-2u}$ is not locally integrable near $o$)
and Proposition \ref{Pro:GZ1005},
it follows that
$$\sup\{c_{o}^{\tilde{I}}(\tilde{u})|\tilde{I}\subseteq\mathcal{O}_{o}\ \& \ \tilde{I}|_{H}=I\}\leq 1,$$
where $\tilde{u}:=\max\{u,\frac{1}{l-1}\log\sum_{j}|g_{j}|\}$.
Since $\tilde{u}\geq u$,
which implies
$$\sup\{c_{o}^{\tilde{I}}(\tilde{u})|\tilde{I}\subseteq\mathcal{O}_{o}\ \& \ \tilde{I}|_{H}=I\}\geq \sup\{c_{o}^{\tilde{I}}(u)|\tilde{I}\subseteq\mathcal{O}_{o}\ \& \ \tilde{I}|_{H}=I\}=1,$$
then it follows that $\sup\{c_{o}^{\tilde{I}}(\tilde{u})|\tilde{I}\subseteq\mathcal{O}_{o}\ \& \ \tilde{I}|_{H}=I\}=1$.

By the restriction formula on jumping numbers,
it follows that
$\sup\{c_{o}^{\tilde{I}}(\tilde{u})|\tilde{I}\subseteq\mathcal{O}_{o}\ \& \ \tilde{I}|_{H}=I\}\geq c_{o'}^{I}(\tilde{u}|_{H})\geq c_{o'}^{I}(u|_{H})=1$
($\Leftarrow \tilde{u}|_{H}\geq u|_{H}$),
which implies
$c_{o'}(\tilde{u}|_{H})=\sup\{c_{o}^{\tilde{I}}(\tilde{u})|\tilde{I}\subseteq\mathcal{O}_{o}\ \& \ \tilde{I}|_{H}=I\}=1.$

Let $Y:=Supp\{z|c_{z}(\tilde{u})\leq 1\}=Supp(\mathcal{O}/\mathcal{I}(c_{o'}^{I}(u|_{H})u))$.
We prove Theorem \ref{t:jump_equality} by contradiction.
If not, then $dim Y<n-k$.
By Lemma \ref{l:intersec_fibre},
there exists a $k+1$ dimensional plane $H_{1}\supset H$
such that $H_{1}\cap Y= H\cap Y$ (without loss of generality, one can retract the $\Delta^{n}$).
By changing of the coordinates,
we set $H_{1}:=\{z_{k+2}=\cdots=z_{n}\}.$

By the restriction formula on jumping numbers,
it follows that
$$\sup\{c_{o}^{\tilde{I}}(\tilde{u})|\tilde{I}\subseteq\mathcal{O}_{o}\ \& \ \tilde{I}|_{H}=I\}\geq \sup\{c_{o''}^{\tilde{I}}(\tilde{u}|_{H_{1}})|\tilde{I}\subseteq\mathcal{O}_{o''}\ \& \ \tilde{I}|_{H}=I\}
\geq c_{o'}^{I}(\tilde{u}|_{H}),$$
where $o''\in H_{1}$ is the origin, which emphasizes that $c_{o''}^{\tilde{I}}(\tilde{u}|_{H_{1}})$ is computed on the submanifold $H_{1}$.
As
$$\sup\{c_{o}^{\tilde{I}}(\tilde{u})|\tilde{I}\subseteq\mathcal{O}_{o}\ \& \ \tilde{I}|_{H}=I\}=c_{o'}^{I}(\tilde{u}|_{H})=1,$$
then it follows that
$$\sup\{c_{o''}^{\tilde{I}}(\tilde{u})|\tilde{I}\subseteq\mathcal{O}_{o''}\ \& \ \tilde{I}|_{H}=I\}=c_{o'}(\tilde{u}|_{H})=1.$$

As $H_{1}\cap Y= H\cap Y$,
then $\{z|\mathcal{I}(\tilde{u}|_{H_{1}})_{z''}\neq \mathcal{O}_{z''},z\in H_{1}\}\subseteq H\cap Y$.
By Corollary \ref{Coro:jump=1} on $H_{1}$ ($u\sim\tilde{u}|_{H_{1}}$, $J\sim I$),
it follows that there exists $N>0$ (independent of $\tilde{I}\subseteq\mathcal{O}_{o''}$)
such that
\begin{equation}
\label{equ:20150320a}
\sup_{\tilde{I}|_{H}=I}\{c_{o''}^{\tilde{I}}(\max\{\tilde{u}|_{H_{1}},N\log|z_{k+1}|\})\}\leq1.
\end{equation}

By Corollary \ref{coro:jump>1},
it follows that
$$\sup_{\tilde{I}|_{H}=I}\{c_{o''}^{\tilde{I}}(\max\{\tilde{u}|_{H_{1}},N\log|z_{k+1}|\})\}>1,$$
which contradicts inequality \ref{equ:20150320a}.

Then the present theorem has been proved.

\section{Proofs of the applications of Theorem \ref{t:jump_equality}}\label{sec:proof_applications}

In this section,
we present the proofs of applications of Theorem \ref{t:jump_equality}.

\subsection{Proof of Theorem \ref{thm:jump_equality_dim_sing}}

As $c_{o'}^{I}(u|_{H})=c$ (equality \ref{equ:jump_equ}), $c_{z'}(u|_{H})\leq c$ for any $z'\in (A\cap H)$ (Proposition \ref{prop:DK2000}) and $\dim((A\cap H,o)\setminus(V(I),o))=\dim_{o}(A\cap H)$,
by using Remark \ref{rem:jump_supp} ($o\sim o'$, $u\sim u|_{H}$, $A\sim A\cap H$),
it follows that there exists $z_{0}\in((A\cap H,o')\setminus (V(I),o'))$
such that $c_{z_{0}'}(u|_{H})=c_{o'}^{I}(u|_{H})$ and $\dim_{z_{0}}(A\cap H)=\dim_{o}(A\cap H)$.

Note that $\dim_{o}A\geq\dim_{z_{0}}A$,
then it suffices to consider:
$$"c_{o'}(u|_{H})=c_{o}(u)"\Rightarrow"\dim_{o}A=n-k+\dim_{o}(A\cap H)"$$
($\dim_{o}A\geq\dim_{z_{0}}A=n-k+\dim_{z_{0}}(A\cap H)=n-k+\dim_{o}(A\cap H)$, $z_{0}\sim o$ in the first $"="$).

By Remark \ref{rem:GZ1005} $(u\sim c_{o}(u)u, J=\mathcal{I}(c_{o}(u)u)_{o})$,
it follows that
\begin{equation}
\label{equ:151013b}
c_{z}(\tilde{u})\leq 1
\end{equation}
for any $z\in (A,o)$
and $c_{o}(\tilde{u})=1$ $(\Leftarrow c_{o}(c_{o}(u)u)=1)$.
By Proposition \ref{prop:DK2000},
it follows that
\begin{equation}
\label{equ:151010a}
c_{z'}(\tilde{u}|_{H})\leq 1
\end{equation}
for any $z\in (A\cap H,o)$.

Using Proposition \ref{prop:DK2000} and inequality \ref{equ:151010a},
one can obtain that $c_{o'}(\tilde{u}|_{H})\leq c_{o}(\tilde{u})\leq 1$.
Combining with
$1=c_{o}(u)/c_{o}(u)=c_{o'}(u|_{H})/c_{o}(u)=c_{o'}(c_{o}(u)u|_{H})\leq c_{o'}(\tilde{u}|_{H})$ $(\Leftarrow c_{o}(u)u\leq\tilde{u})$,
one can obtain that
\begin{equation}
\label{equ:151013a}
c_{o'}(\tilde{u}|_{H})=c_{o}(\tilde{u})=1.
\end{equation}

Let $l=k-\dim_{o}(A\cap H)$.
Let $A_{3}$ be a irreducible component of
$A\cap H$ on $\mathbb{B}^{k-l}\times\mathbb{B}^{l}\subset H$ through $o$
satisfying $\dim_{o}A_{3}=k-l$.

By the parametrization of $(A_{3},o)$ in $H$ (see "Local parametrization theorem" (4.19) in \cite{demailly-book}),
it follows that
one can find local coordinates $(z_{1},\cdots,z_{n})$ of
a neighborhood $U=\mathbb{B}^{k-l}\times\mathbb{B}^{l}\times\mathbb{B}^{n-k}$ of $o$ satisfying $H=\{z_{k+1}=\cdots=z_{n}=0\}$
and $\dim(A\cap U)=\dim_{o}A$
such that

$(1)$ $A_{3}\cap((\mathbb{B}^{k-l}\times\mathbb{B}^{l})\cap H)$ is reduced and irreducible;

$(2)$ for any $a=(a_{1},\cdots,a_{k-l})\in\mathbb{B}^{k-l}$,
$A_{3}\cap L_{a}\neq\emptyset$, where $L_{a}=\{z_{1}=a_{1},\cdots,z_{k-l}=a_{k-l}\}$;

$(3)$ there exists analytic subset $A_{4}\subseteq\mathbb{B}^{k-l}$
such that any $z\in (A_{3}\cap p^{-1}(\mathbb{B}^{k-l}\setminus A_{4}))$
is the regular point in $A_{3}$ and the noncritical point of $p|_{A_{3,reg}}$,
where $p:(z_{1},\cdots,z_{k})=(z_{1},\cdots,z_{k-l})$.

By $(2)$, $(3)$, $c_{o}(\tilde{u})=c_{o'}(\tilde{u}|_{H})=1$ (inequality \ref{equ:151013a}),
$c_{z}(\tilde{u})\leq1$ for any $z\in A_{3}$ (inequality \ref{equ:151013b}),
and Remark \ref{rem:lct_slice_sing},
it follows that for almost every
$a=(a_{1},\cdots,a_{k-l})\in\mathbb{B}^{k-l}$
with respect to the Lebesgue measure on $\mathbb{C}^{k-l}$,
there exists $z_{a}\in A_{3}\cap L_{a}$
such that
equality $c_{z_{a}'}(\tilde{u}|_{L_{a}})=1$ (the set of $a$ denoted by $A_{ae}$),
where $z_{a}'$ emphasizes that $c_{z_{a}'}(\tilde{u}|_{L_{a}})$ is computed on the submanifold $L_{a}$.

Let $\tilde{L}_{a}=\{z_{1}=a_{1},\cdots,z_{k-l}=a_{k-l}\}$.
By inequality \ref{equ:151013b} and Proposition \ref{prop:DK2000},
it follows that for any $a\in A_{ae}$,
$1=c_{z_{a}'}(\tilde{u}|_{L_{a}})\leq c_{z_{a}''}(\tilde{u}|_{\tilde{L}_{a}})\leq c_{z_{a}}(\tilde{u})=1$,
which implies $c_{z_{a}'}(\tilde{u}|_{L_{a}})=c_{z_{a}''}(\tilde{u}|_{\tilde{L}_{a}})=1$,
where
$z_{a}''$ emphasizes that $c_{z_{a}''}(\tilde{u}|_{\tilde{L}_{a}})$ is computed on the submanifold $\tilde{L}_{a}$.

Using
Corollary \ref{coro:lct_restrict_equ}
($\mathbb{C}^{n}\sim\tilde{L}_{a}$,
$H\sim H\cap\tilde{L}_{a}=L_{a}$,
$o\sim z_{a}$, $u\sim\tilde{u}|_{\tilde{L}_{a}}$),
one can obtain that
for any $a\in A_{ae}$,
$\max_{z_{a}\in p^{-1}(a)}\dim_{z_{a}}\{z''|c_{z''}(\tilde{u}|_{\tilde{L}_{a}})\leq 1\}\geq n-l-(k-l)=n-k$.
By the definition of $\tilde{u}$,
it follows that
$((A\cap U)\cap \tilde{L}_{a})\supseteq\{z''|c_{z''}(\tilde{u}|_{\tilde{L}_{a}})\leq 1\}$,
which implies
$\dim((A\cap U)\cap \tilde{L}_{a})\geq\max_{z_{a}\in p^{-1}(a)}\dim_{z_{a}}\{z''|c_{z''}(\tilde{u}|_{\tilde{L}_{a}})\leq 1\}$.
Then we obtain that
the $2(n-k)$-dimensional Hausdorff measure
of $\dim((A\cap U)\cap \tilde{L}_{a})$ is not zero for any $a\in A_{ae}$.

Note that the $2(k-l)$-dimensional Hausdorff measure of $A_{ae}$ is not zero,
then it follows that the $2(n-k)+2(k-l)=2(n-l)$-dimensional Hausdorff measure
of $A$ near $o$ is not zero (see Theorem 3.2.22 in \cite{Fed69}),
which implies that
$\dim_{o}A=\dim(A\cap U)\geq n-l$.
Note that $l=k-\dim_{o}(A\cap H)$ implies $\dim_{o}A\leq n-k+(k-l)=n-l$,
then Theorem \ref{thm:jump_equality_dim_sing} has been proved.

\subsection{Proof of Theorem \ref{thm:lct_slice_graph}}
By Corollary \ref{coro:jump_supp},
it follows that $(2)\Leftrightarrow (3)$.

In order to prove Theorem \ref{thm:lct_slice_graph},
by Theorem \ref{thm:jump_equality_dim_sing},
it suffices to prove the following statement $((1)\Rightarrow(2))$.\\

Assume that $(A\cap H,o)$ is regular,
and
$k-\dim_{o}A\cap H=n-\dim_{o}A$.
If $c_{o}(u)=c_{o}(u|_{H})$,
then there exist coordinates $(w_{1},\cdots,w_{k},z_{k+1},\cdots,z_{n})$ near $o$ and $l\in\{1,\cdots,k\}$,
such that $(w_{1}=\cdots=w_{l}=0,o)=(A,o).$\\

Let $J_{0}=\mathcal{I}(c_{o}(u)u)_{o}$.
By Remark \ref{rem:GZ1005} $(u\sim c_{o}(u)u)$,
it follows that there exists $p_{0}>0$ large enough,
such that $\tilde{u}:=\max\{c_{o}(u)u,p_{0}\log|J_{0}|\}$
satisfies:
$(1)$ $c_{o}(\tilde{u})=1$ $(\Leftarrow c_{o}(c_{o}(u)u)=1)$;
$(2)$ $(\{z|c_{z}(\tilde{u})\leq 1\},o)=(A,o)$.

By $\tilde{u}|_{H}\geq c_{o}(u)u|_{H}= c_{o'}(u|_{H})u|_{H}$,
it follows that
$c_{o'}(\tilde{u}|_{H})\geq c_{o'}(c_{o}(u)u|_{H})=c_{o'}(c_{o'}(u|_{H})u|_{H})=1$.
Combining with the fact that
$c_{o'}(\tilde{u}|_{H})\leq c_{o}(\tilde{u})=1$,
we obtain that
\begin{equation}
\label{equ:151009a}
c_{o'}(\tilde{u}|_{H})=1.
\end{equation}

Note that $c_{z'}(\tilde{u}|_{H})\leq c_{z}(\tilde{u})$ for any $z\in A\cap H$,
then by $(2)$ $(\Rightarrow c_{z}(\tilde{u})\leq 1)$ for any $z\in A\cap H$,
it follows that
$(\{z|c_{z'}(\tilde{u}|_{H})\leq 1\},o)\supseteq(A\cap H,o).$
Combining with the definition of
$\tilde{u}$ $(\Rightarrow (\{z|c_{z'}(\tilde{u}|_{H})<+\infty\},o)\subseteq(V(J_{o})\cap H,o)=(A\cap H,o))$,
we obtain
\begin{equation}
\label{equ:151009b}
(V(\mathcal{I}(\tilde{u}|_{H})),o)=(\{z|c_{z'}(\tilde{u}|_{H})\leq 1\},o)=(A\cap H,o).
\end{equation}

In the following part of the present section,
we consider $\tilde{u}$ instead of $u$.

By equality \ref{equ:151009b}, it follows that $(V(\mathcal{I}(\tilde{u}|_{H})),o)(=(A\cap H,o))$ is regular.
Combining with equality \ref{equ:151009a} and Corollary \ref{coro:jump_supp} $(u\sim \tilde{u}|_{H})$,
it follows that there exist $l\in\{1,\cdots,k\}$ and holomorphic functions $f_{1},\cdots,f_{l}$ near $o'\in H$ such that \\
$(a)$ $df_{1}|_{o'},\cdots,df_{l}|_{o'}$ are linear independent;\\
$(b)$ $(\{f_{1}=\cdots=f_{l}=0\},o)=(A\cap H,o)$ holds;\\
$(c)$ $|f_{j}|^{2}e^{-2\tilde{u}|_{H}}$ are all locally integrable near $o'$ for $j\in\{1,\cdots,l\}$.

By Remark \ref{rem:ot_plane} and $(c)$,
it follows that there exist holomorphic functions $F_{1},\cdots,F_{l}$ near $o\in \mathbb{C}^{n}$ such that
 and $|F_{j}|^{2}e^{-2\tilde{u}}$ are integrable near $o$ for any $j\in\{1,\cdots,l\}$,
which implies that $\{F_{1}=\cdots=F_{l}=0\}\supseteq A$.
Combining $F_{j}=f_{j}$ and $(a)$, we obtain that $dF_{1}|_{o},\cdots,dF_{l}|_{o},dz_{k+1}|_{o},\cdots,dz_{n}|_{o}$ are linear independent.

Note that $\{F_{1}=\cdots=F_{l}=0\}$ is regular near $o$ and $n-\dim_{o}A=k-\dim_{o}A\cap H=l$,
then it follows that $\{F_{1}=\cdots=F_{l}=0\}=A$ near $o$.
Choosing $w_{j}=F_{j}$ for any $j\in\{1,\cdots,l\}$,
one can find holomorphic functions $w_{l+1},\cdots,w_{k}$ near $o$ such that
$dw_{1}|_{o},\cdots,dw_{k}|_{o},dz_{k+1}|_{o},\cdots,dz_{n}|_{o}$ are linear independent.
Then Theorem \ref{thm:lct_slice_graph} has been proved.

\subsection{Proof of Remark \ref{prop:add_dim_nonregular}}\label{sec:add_nonregular}

Let $A_{1}=V(\mathcal{I}(cu))$ and $A_{2}=V(\mathcal{I}(cv))$,
and $A=\{(z,w)|c_{(z,w)}(\max\{\pi_{1}^{*}(u),\pi_{2}^{*}(v)\})\leq c\}$.
By Proposition \ref{thm:subadd_cse_general},
it follows that
\begin{equation*}
\begin{split}
c_{(o,o)}(\max\{\pi_{1}^{*}(u),\pi_{2}^{*}(v)\})
&=c_{o}(u)+c_{o}(v)=c
\\&=c_{o}(\max\{u,v\})
=c_{(z,w)}(\max\{\pi_{1}^{*}(u),\pi_{2}^{*}(v)\}|_{\Delta}).
\end{split}
\end{equation*}
Using Theorem \ref{thm:jump_equality_dim_sing} $(u\sim\max\{\pi_{1}^{*}(u),\pi_{2}^{*}(v)\}$,
$H\sim\Delta$, $o\sim (o,o)$, $k\sim n$, $n\sim 2n)$,
we obtain $\dim_{(o,o)}A=\dim_{(o,o)}(A\cap\Delta)+n$.
By Proposition \ref{thm:add_prod_cse},
it follows that
$A=\{(z,w)|c_{z}(u)+c_{w}(v)\leq c\}\subseteq\{(z,w)|\max\{c_{z}(u),c_{w}(v)\}\leq c\}=A_{1}\times A_{2}$,
which implies  $\dim_{o}A_{1}+\dim_{o}A_{2}=\dim_{(o,o)}(A_{1}\times A_{2})\geq \dim_{(o,o)}A$.
Note that $B=\{z|c_{z}(u)+c_{z}(v)\leq c\}$ is biholomophic to $A\cap\Delta$,
then it follows that $\dim_{o}A_{1}+\dim_{o}A_{2}\geq \dim_{(o,o)}A=\dim_{(o,o)}(A\cap\Delta)+n=n+\dim_{o}B$.
Remark \ref{prop:add_dim_nonregular} has thus been proved.

\subsection{Proof of Proposition \ref{prop:lct_add_graph}}
Following the symbols in subsection \ref{sec:add_nonregular},
by Theorem \ref{thm:lct_slice_graph} $(n\sim 2n$, $k\sim n$, $u\sim\max\{\pi_{1}^{*}u,\pi_{2}^{*}v\}$, $o\sim (o,o)\in\mathbb{C}^{n}\times\mathbb{C}^{n}$,
$H\sim\Delta$ the diagonal of $\mathbb{C}^{n}\times\mathbb{C}^{n})$,
it follows that $A$ is regular at $((o,o))$ satisfying $\dim_{(o,o)}A=\dim_{(o,o)}(A\cap\Delta)+n$.
As $A_{1}\cap A_{2}=B$,
it follows that
$(A_{1}\times A_{2})\cap\Delta=A\cap\Delta$,
which implies
\begin{equation}
\label{equ:151024a}
\dim_{(o,o)}A
=\dim_{(o,o)}(A\cap\Delta)+n
=\dim_{(o,o)}((A_{1}\times A_{2})\cap\Delta)+n.
\end{equation}

Note that $A_{1}\times A_{2}\supseteq A$ and equality \ref{equ:151024a} holds,
then it follows that
$\dim_{(o,o)}(A_{1}\times A_{2})\geq\dim_{(o,o)}A=\dim_{(o,o)}((A_{1}\times A_{2})\cap\Delta)+n$.
As $\Delta$ is regular,
then it is clear that
$\dim_{(o,o)}(A_{1}\times A_{2})\leq\dim_{(o,o)}((A_{1}\times A_{2})\cap\Delta)+n$,
which implies
$\dim_{(o,o)}(A_{1}\times A_{2})=\dim_{(o,o)}((A_{1}\times A_{2})\cap\Delta)+n=\dim_{(o,o)}A$.
Note that $(A,(o,o))$ is regular and $A_{1}\times A_{2}$ is irreducible at $(o,o)$ ($A_{1}$ and $A_{2}$ are both irreducible at $o$),
then we obtain $A=A_{1}\times A_{2}$,
which implies $A_{1}$ and $A_{2}$ are both regular.

By the transversality between $A_{1}\times A_{2}=A$ and $\Delta$ at $(o,o)$ and Remark \ref{lem:linear},
it follows that $2n=\dim(T_{A_{1}\times A_{2},(o,o)}+T_{\Delta,(o,o)})=\dim T_{A_{1}\times A_{2},(o,o)}+\dim T_{\Delta,(o,o)}-\dim(T_{A_{1}\times A_{2},(o,o)}\cap T_{\Delta,(o,o)})
=(\dim T_{A_{1},o}+\dim T_{A_{2},o})+n-\dim(T_{A_{1},o}\cap T_{A_{2},o})=\dim(T_{A_{1},o}+ T_{A_{2},o})+n$.
It is clear that $\dim(T_{A_{1},o}+T_{A_{2},o})=n$,
then we prove Proposition \ref{prop:lct_add_graph}.

\section{Proofs of two sharp relations on jumping numbers}\label{sec:proof_relations}

In the present section, we prove Theorem \ref{t:GZ_jump_sharp} and Theorem \ref{thm:sharp_jump_inequ}.

\subsection{Proof of Theorem \ref{t:GZ_jump_sharp}}
$\\$

By $c_{o}^{I}(\tilde{u}_{l})=1$ in Remark \ref{rem:GZ1005},
it follows that
$c_{o}^{I}(\max\{c^{I}_{o}(u)u,\frac{1}{\frac{c^{IJ}_{o}(u)}{c^{I}_{o}(u)}-1}\log|J|\})=1$.

By the monotonicity of complex singularity exponents
($u\leq v\Rightarrow c_{o}(u)\leq c_{o}(v)$),
it follows that
$c_{o}^{I}(\frac{1}{\frac{c^{IJ}_{o}(u)}{c^{I}_{o}(u)}-1}\log|J|\})\leq 1$,
i.e.,
\begin{equation}
\label{equ:20150125e}
\frac{c^{I}_{o}(u)}{c^{IJ}_{o}(u)-c^{I}_{o}(u)}\geq c_{o}^{I}(\log|J|).
\end{equation}
Then Theorem \ref{t:GZ_jump_sharp} has thus been proved.\\

For the sake of completeness,
we give a proof of the following equivalence
$$IJ\subseteq\mathcal{I}(c^{I}_{o}(u)u)_{o}\Leftrightarrow c^{IJ}_{o}(u)>c^{I}_{o}(u).$$

Firstly, we  prove "$\Rightarrow$".
Since $IJ\subseteq\mathcal{I}(c^{I}_{o}(u)u)_{o}$ implies that
$|IJ|^{2}e^{-2c^{I}_{o}(u)u}$ is locally integrable near $o$,
then it follows that $c^{IJ}_{o}(u)>c^{I}_{o}(u)$ by the strong openness property.

Secondly, we prove "$\Leftarrow$".
Since $IJ\not\subseteq\mathcal{I}(c^{I}_{o}(u)u)_{o}$ implies that
$|IJ|^{2}e^{-2c^{I}_{o}(u)u}$ is not locally integrable near $o$,
then it follows that $c^{IJ}_{o}(u)\leq c^{I}_{o}(u)$ by the definition of $c^{IJ}_{o}(u)$.

\subsection{Proof of Corollary \ref{coro:GZ_jump_sharp}}
$\\$

When $I=\mathcal{O}_{o}$, inequality \ref{equ:20150125e} degenerates to
\begin{equation}
\label{equ:20150125d}
\frac{c_{o}(u)}{c^{J}_{o}(u)-c_{o}(u)}\geq c_{o}(\log|J|),
\end{equation}
i.e.,
\begin{equation}
\label{equ:20150125f}
c^{J}_{o}(u)\leq\frac{c_{o}(u)}{c_{o}(\log|J|)}+c_{o}(u).
\end{equation}
Then Corollary \ref{coro:GZ_jump_sharp} follows for $J\subseteq\mathcal{I}(c_{o}(u)u)_{o}$.

If $J$ does not satisfy $J\subseteq\mathcal{I}(c_{o}(u)u)_{o}$,
then $|J|^{2}e^{-2c_{o}(u)u}$ is not integrable near $o$,
which implies
$c^{J}_{o}(u)\leq c_{o}(u)$.
Note that $|J|^{2}$ is locally bounded near $o$.
Then it follows that
$|J|^{2}e^{-2cu}$ is locally integrable near $o$ for any $c<c_{o}(u)$,
which implies
$c^{J}_{o}(u)\geq c_{o}(u)$.
Then it is clear that $c^{J}_{o}(u)=c_{o}(u)$.

Then  Corollary \ref{coro:GZ_jump_sharp} has been proved.

\subsection{Proof of Theorem \ref{thm:sharp_jump_inequ}}
$\\$

By Corollary \ref{coro:jump_compare20150424},
it follows that for any $\varepsilon>0$,
there exists a coherent ideal $\tilde{I}|_{H}=I$
such that
\begin{equation}
\label{equ:20150531a}
c_{o}^{\tilde{I}}(\max(\varphi,\frac{N-1}{b}\log|h|))\geq  \frac{bN}{N-1}-\varepsilon.
\end{equation}

Assume that $b_{1}>0.$
For any $\tilde{I}|_{H}=I$, it follows that
\begin{equation}
\label{equ:20150515b}
\frac{1}{c_{o}^{\tilde{I}h}(\varphi)-c_{o}^{\tilde{I}}(\varphi)}\leq \frac{1}{b_{1}}.
\end{equation}
By Remark \ref{rem:GZ1005} ($l\sim\frac{c_{o}^{\tilde{I}h}(\varphi)}{c_{o}^{\tilde{I}}(\varphi)}$,
$I\sim\tilde{I}$, $J\sim h$),
it follows that
\begin{equation}
\label{equ:20150513a}
c_{o}^{\tilde{I}}(\max\{c_{o}^{\tilde{I}}(\varphi)\varphi,\frac{1}{\frac{c_{o}^{\tilde{I}h}(\varphi)}{c_{o}^{\tilde{I}}(\varphi)}-1}\log|h|\})=1.
\end{equation}

By Corollary \ref{coro:jump_compare20150424}
($k\sim n-1$,
$z_{k+1}\sim h$, $b\sim c_{o'}^{I}(\varphi|_{H})$,
$\frac{N-1}{b}\sim\frac{1}{b_{1}}$),

which deduces
\begin{equation}
\label{equ:20150513c}
\begin{split}
&\sup_{\tilde{I}|_{H}=I}c_{o}^{\tilde{I}}(\max\{\varphi,\frac{1}{b_{1}}\log|h|\})
\geq
\frac{c_{o'}^{I}(\varphi|_{H})(c_{o'}^{I}(\varphi|_{H})+b_{1})}
{c_{o'}^{I}(\varphi|_{H})}
=
c_{o'}^{I}(\varphi|_{H})+b_{1}
\end{split}
\end{equation}

By equality \ref{equ:20150513a},
it follows that
$$c_{o'}^{I}(\varphi|_{H})+b_{1}\leq b_{0}$$
holds for any $\varepsilon>0$
(if not, then $">"$ holds.
Combining with inequality \ref{equ:20150515b},
it follows that there exists $\tilde{I}$ such that
\begin{equation}
\label{equ:20150515a}
\begin{split}
&c_{o}^{\tilde{I}}(\max\{\varphi,(\frac{1}{c_{o}^{\tilde{I}h}(\varphi)-c_{o}^{\tilde{I}}(\varphi)})\log|h|\})
\geq c_{o}^{\tilde{I}}(\max\{\varphi,\frac{1}{b_{1}}\log|h|\})
>b_{0}\geq c_{o}^{\tilde{I}}(\varphi)
\end{split}
\end{equation}
that contradicts equality \ref{equ:20150513a}).
Then Theorem \ref{thm:sharp_jump_inequ} is proved for the case $b_{1}>0$.

When $b_{1}\leq 0$,
noting that
$c_{o'}^{I}(\varphi|_{H})\leq b_{0}$ (Restriction formula (jumping number)),
we prove Theorem \ref{thm:sharp_jump_inequ}.

\subsection{Proof of Corollary \ref{coro:sharp_lct}}
$\\$

By Remark \ref{rem:slicing},
it suffices to consider $k=n-1$.

Consider the holomorphic map $p:\mathbb{C}^{n}\to \mathbb{C}^{n}$ with coordinates
$(z_{1},\cdots,z_{n})$ and $(w_{1},\cdots,w_{n})$ respectively
satisfying $p(z_{1},\cdots,z_{n-1},z_{n})=(z_{1},\cdots,z_{n-1},z_{n-1}z_{n})$.
Then it follows that
$$\int_{\Delta^{n}_{r}}e^{-2l\varphi}=\int_{\Delta_{r}(a)}\int_{\Delta^{n}_{r}\cap \{z_{n}=a\}}|z_{n-1}|^{2}e^{-2l\varphi\circ p},$$
which implies
$c^{w_{n-1}}_{o'}(\varphi|_{\frac{w_{n}}{w_{n-1}}=a})
=c^{z_{n-1}}_{(0,\cdots,0,a)'}((\varphi\circ p)|_{z_{n}=a})\geq c_{n}$
for a.e. $a\in\Delta_{r}$ ($r>0$ small enough, using lower semicontinuity of complex singularity exponent),
where $(0,\cdots,0,a)'$ emphasizes that $c^{z_{n-1}}_{(0,\cdots,0,a)'}((\varphi\circ p)|_{\{z_{n}=a\}})$ is computed on the submanifold $\{z_{n}=a\}$.

Using inequality \ref{equ:sharp_lct} ($n\sim n-1$, $h(=z_{n})\sim w_{n-1}$, $\varphi\sim\varphi|_{\frac{w_{n}}{w_{n-1}}=a}$,
$c^{h}_{o}(\varphi)\sim c^{w_{n-1}}_{o'}(\varphi|_{\frac{w_{n}}{w_{n-1}}=a})$),
we obtain
Corollary \ref{coro:sharp_lct}.

\begin{Remark}
\label{rem:slicing}
For any $k\in\{1,\cdots,n-2\}$, there exist dim $k$ and $k+2$ planes $H_{k}$ and $H_{k+2}$ planes through $o$
satisfying $H_{k}\subset H_{k+2}\subset \mathbb{C}^{n}$,
such that
$c_{o'}(\varphi|_{H_{k}})=c_{k}$ and $c_{o''}(\varphi|_{H_{k+2}})=c_{k+2}$,
where $o''$ emphasizes that $c_{o''}(\varphi|_{H_{k+2}})$ is computed on the submanifold $H_{k+2}$.
\end{Remark}

By Remark \ref{rem:lct_slice},
it suffices to consider the following remark (proof see Section \ref{sec:appendix II}).

\begin{Remark}
\label{rem:Grassman}
Let $G_{1}$ and $G_{2}$ be two subsets of $G(k_{1},n)$ and $G(k_{2},n)$ whose complements are of $U(n)$-invariant measure $0$ respectively
($k_{1}<k_{2}$).
Then there exists $V_{1}\in G_{1}$ and $V_{2}\in G_{2}$ satisfying $V_{1}\subset V_{2}$.
\end{Remark}

\subsection{Proof of Remark \ref{rem:Grassman}}\label{sec:appendix II}
$\\$

It is well-known that
a Zariski open set of $G(k_{1},n)$ ($G(k_{2},n)$) could be presented as $(M(k_{1},n-k_{1})\,(\delta_{j,k_{1}+1-l})_{1\leq j,l\leq k_{1}})$
($(M(k_{2},n-k_{2})\,(\delta_{j,k_{2}+1-l})_{1\leq j,l\leq k_{2}})$) with respect to the same coordinate
$(z_{1},\cdots,z_{n})$.

We consider two cases: (1) $n\geq k_{1}+k_{2}$; (2) $n<k_{1}+k_{2}$.

$\\$
\textbf{proof of Case (1)}

Let $D\in M(k_{2}-k_{1},n-(k_{1}+k_{2}))$.
Consider a family of mappings $p_{D}$ from $M(k_{1},n-k_{1})$ to $(M(k_{2},n-k_{2}),(\delta_{j,k_{2}+1-l})_{1\leq j\leq k_{2},1\leq l\leq k_{2}-k_{1}})$:

\begin{equation}
\label{equ:linearA}
p_{D}(A\,B):=\left(
          \begin{array}{cc}
            A-B\times (\delta_{j,k_{2}-k_{1}+1-l})_{1\leq j,l\leq k_{2}-k_{1}}\times(D\,B^{t}) & 0 \\
            (D\,B^{t}) & (\delta_{j,k_{2}-k_{1}+1-l})_{1\leq j,l\leq k_{2}-k_{1}}  \\
          \end{array}
        \right)
\end{equation}
i.e.,
\begin{equation}
\label{equ:linearB}
\left(
          \begin{array}{cc}
           (\delta_{j,l})_{1\leq j,l\leq k_{1}}  & -B\times (\delta_{j,k_{2}-k_{1}+1-l})_{1\leq j,l\leq k_{2}-k_{1}}  \\
            0 & (\delta_{j,l})_{1\leq j,l\leq k_{2}-k_{1}}  \\
          \end{array}
        \right)
\left(
          \begin{array}{cc}
            A & B \\
            (D,B^{t}) & (\delta_{j,k_{2}-k_{1}+1-l})_{1\leq j,l\leq k_{2}-k_{1}}  \\
          \end{array}
        \right),
\end{equation}
for any $A\in M(k_{1},n-k_{2})$ and $B\in M(k_{1},k_{2}-k_{1})$.

Note that

(1a) for any $D$, holomorphic map $p_{D}$ is injective;

(1b) $\sqcup_{D}p_{D}(M(k_{1},n-k_{1}))=(M(k_{2},n-k_{2})\,(\delta_{j,k_{2}+1-l})_{1\leq j\leq k_{2},1\leq l\leq k_{2}-k_{1}})$,
which implies that for a.e. $D\in M(k_{2}-k_{1},n-(k_{1}+k_{2}))$ and a.e. $M\in M(k_{1},n-k_{1})$
$(P_{D}(M)\,(\delta_{j,k_{1}+1-l})_{1\leq j,l\leq k_{1}})\in G_{2}$.

(1c) the vector space generated by the row vector of $(p_{D}(A\,B)\,(\delta_{j,k_{1}+1-l})_{1\leq j, l\leq k_{1}})$
contains the vector space generated by the row vector of
$(A\,B\,(\delta_{j,k_{1}+1-l})_{1\leq j\leq k_{2},1\leq l\leq k_{1}})$ (by equality \ref{equ:linearB}).

By (1b) and (1c), it follows that case (1) has been proved.

$\\$
\textbf{Proof of Case (2)}

Let $D\in M(k_{2}-k_{1},k_{1}+k_{2}-n)$.
Consider a family of mappings $p_{D}$ from subset $G_{D}:=\{A\,(D\,B^{t})^{t}|A\in M(k_{1},n-k_{2}),B\in M(n-k_{2},k_{2}-k_{1})\}$
of $M(k_{2},n-k_{2})$ to $(M(k_{1},n-k_{1})\,(\delta_{j,k_{2}+1-l})_{1\leq j\leq k_{2},1\leq l\leq k_{2}-k_{1}})$:

\begin{equation}
\label{equ:linearC}
p_{D}( A\,(D\,B^{t})^{t}):=\left(
          \begin{array}{cc}
            A-(D\,B^{t})^{t}\times(\delta_{j,k_{2}-k_{1}+1-l})_{1\leq j,l\leq k_{2}-k_{1}}\times B^{t} & 0\\
            B^{t}  & (\delta_{j,k_{2}-k_{1}+1-l})_{1\leq j,l\leq k_{2}-k_{1}}  \\
          \end{array}
        \right)
\end{equation}
i.e.,
\begin{equation}
\label{equ:linearD}
\left(
          \begin{array}{cc}
           (\delta_{j,l})_{1\leq j,l\leq k_{1}}  &  -(D\,B^{t})^{t}\times(\delta_{j,k_{2}-k_{1}+1-l})_{1\leq j,l\leq k_{2}-k_{1}} \\
           0                                     &  (\delta_{j,l})_{k_{1}+1\leq j,l\leq k_{2}}
          \end{array}
        \right)
\left(
          \begin{array}{cc}
            A & (D\,B^{t})^{t} \\
           B^{t}  & (\delta_{j,k_{2}-k_{1}+1-l})_{1\leq j,l\leq k_{2}-k_{1}}  \\
          \end{array}
        \right),
\end{equation}
for any $A\in M(k_{1},n-k_{2})$ and $B\in M(n-k_{2},k_{2}-k_{1})$.

Note that

(2a) for any $D$, holomorphic map $p_{D}$ is surjective and injective;

(2b) $\sqcup_{D}G_{D}=(M(k_{2},n-k_{2})\,(\delta_{j,k_{2}+1-l})_{1\leq j\leq k_{2},1\leq l\leq k_{2}-k_{1}})$,
which implies that for a.e. $D\in M(k_{2}-k_{1},k_{1}+k_{2}-n)$ and a.e. $M\in G_{D}$,
$(P_{D}(M)\,(\delta_{j,n-l})_{1\leq j,l\leq k_{1}})\in G_{1}$.

(2c) the vector space generated by the row vector of $(p_{D}( A\,(D\,B^{t})^{t})\,(\delta_{j,k_{1}+1-l})_{1\leq j\leq k_{2},1\leq l\leq k_{1}})$ contains the vector space generated by the row vector of
$(A\,(D\,B^{t})^{t}\,(\delta_{j,k_{1}+1-l})_{1\leq j, l\leq k_{1}})$ (by equality \ref{equ:linearD}).

By (2b) and (2c), it follows that case (2) has been proved.

\section{Berndtsson's log subharmonicity and integrability}\label{sec:Berndtsson}

In this section,
we present a relationship between
Berndtsson's log subharmonicity and integrability.

We recall a lemma which was used
in \cite{GZopen-a,GZopen-b,GZopen-c} to prove Demailly's strong openness conjecture:

\begin{Lemma}
\label{l:open_b}(see \cite{GZopen-a,GZopen-b})
Let $h_{a}$ be a holomorphic function on unit disc $\Delta\subset\mathbb{C}$
which satisfies
$h_{a}(o)=0$ and $h_{a}(a)=1$ for any $a$,
then we have
$$\int_{\Delta_{r}}|h_{a}|^{2}d\lambda_{1}>C_{1}|a|^{-2},$$
where $a\in\Delta$ whose norm is smaller than $\frac{1}{6}$,
$C_{1}$ is a positive constant independent of $a$ and $h_{a}$.
\end{Lemma}

Let $u$ be a plurisubharmonic function on $\Delta^{n}\times\Delta^{m}$ $(n=k,m=1)$
with coordinates $(z_{1},\cdots,z_{k},w)$,
and $p$ be the projection with $p(z_{1},\cdots,z_{k},w)=w$
and $K_{2u}$ be the fiberwise Bergman kernel as in the above subsection.

\begin{Proposition}
\label{prop:bergman20150130}
If $u>0$,
then $e^{-2u}$ is integrable near the origin $(o,o_{w})\in\mathbb{C}^{k+1}$ if and only if
$K_{2u}^{-1}(o,w)$
is integrable near the origin $o_{w}$ with respect to $w$,
i.e.,
$$\nu( \frac{1}{2}\log K_{2u}(o,\cdot),o_{w})\geq 1.$$
\end{Proposition}

\begin{proof}
It is clear that if
$e^{-2u}$ is integrable near origin $(o,o_{w})$,
then $K_{2u}^{-1}(o,w)$ is integrable near $o_{w}$.
Then it suffices to prove "only if" part,
i.e., if $e^{-2u}$ is not integrable near $(o,o_{w})$,
then $K_{2u}^{-1}(o,w)$ is not integrable near $o_{w}$.

We use our idea of movably using Ohsawa-Takegoshi $L^{2}$ extension theorem (\cite{GZopen-a,GZopen-b,GZopen-c})
to prove "only if" part:

As $e^{-2u}$ is not integrable near $o$,
then it follows from Theorem \ref{t:ot_plane} $(H=p^{-1}(a))$ that for any $a\in\Delta$,
there exists holomorphic function $F_{a}$ on $\Delta^{k+1}$
such that

$(1)$ $F_{a}(o,a)=1$;

$(2)$ $\int_{\Delta^{k+1}}|F_{a}|^{-2u}\leq C_{D} K^{-1}_{2u}(o,a)$;

$(3)$ $F_{a}(o,o_{w})=0$.

(Using the definition of $K_{2u}$,
one can choose holomorphic $f_{a}$ on $\Delta^{k}\times\Delta$ satisfying
$f_{a}(o,a)=1$ and
$\int_{p^{-1}(a)}|f_{a}|^{2}e^{-2u}=K^{-1}_{2u}(o,a),$
and $F_{a}$ is the Ohsawa-Takegoshi $L^{2}$ extension of $f_{a}$.)

By Lemma \ref{l:open_b} ($F_{a}(z_{1},\cdots,z_{k},\cdot)=h_{a}(\cdot)$) and
the submean inequality of $|F_{a}|^{2}$,
it follows that
$\int_{\Delta^{k+1}}|F_{a}|^{2}>C_{2}\frac{1}{|a|^{2}},$
where $C_{2}>0$ is independent of $a$.
As $u>0$,
then it follows from $(2)$ that
$K^{-1}_{2u}(o,a)\geq\frac{1}{C_{D}}\int_{\Delta^{k+1}}|F_{a}|^{2}e^{-2u}>\frac{C_{2}}{C_{D}}\frac{1}{|a|^{2}}.$
Then the present Proposition has been done.
\end{proof}

\begin{Remark}
\label{lem:integ_Bergman_lelong}
If $e^{-2u|_{p^{-1}(0)}}$ is integrable near $o$,
then $e^{-2u-2c\log|w|}$ is also integrable near $(o,o_{w})$, where $c\in(0,1)$.
\end{Remark}

\begin{proof}
As $e^{-2u|_{p^{-1}(0)}}$ is integrable near $o$,
it follows that
$\nu(\log K_{2u}(o,\cdot),o_{w})=0.$
Since
$K_{2u+2c\log|\cdot|}(o,\cdot)=|\cdot|^{2c}K_{2u}(o,\cdot),$
then
$\nu(\frac{1}{2}\log K_{2u+2c\log|\cdot|}(o,\cdot),o_{w})=c<1.$
By Proposition \ref{prop:bergman20150130}, the present Remark has thus been done.
\end{proof}

%

\bibliographystyle{references}
\bibliography{xbib}

\end{document}